\documentclass{amsart}

\usepackage{amsmath, amsthm, amssymb, amsfonts, mathtools}
\usepackage{graphicx,epsfig,color}
\usepackage{scrpage2}
\usepackage{exscale}

\theoremstyle{plain}

\newtheorem{theorem}{Theorem}[section]

\newtheorem{lemma}[theorem]{Lemma}

\newtheorem{corollary}[theorem]{Corollary}

\newtheorem{proposition}[theorem]{Proposition}

\theoremstyle{definition}

\newtheorem{definition}[theorem]{Definition}

\theoremstyle{remark}
\newtheorem{remark}[theorem]{Remark}

\newcommand{\VK}{\kappa}
\newcommand{\sk}{\kappa}

\DeclareSymbolFont{AMSb}{U}{msb}{m}{n}
\DeclareMathSymbol{\N}{\mathalpha}{AMSb}{"4E}
\DeclareMathSymbol{\R}{\mathalpha}{AMSb}{"52}
\DeclareMathSymbol{\Z}{\mathalpha}{AMSb}{"5A}
\DeclareMathSymbol{\D}{\mathalpha}{AMSb}{"44}
\DeclareMathSymbol{\s}{\mathalpha}{AMSb}{"53}

\newcommand{\sX}{\scriptscriptstyle{X}}
\newcommand{\sM}{\scriptscriptstyle{M}}
\newcommand{\sN}{\scriptscriptstyle{N}}
\newcommand{\sK}{\scriptscriptstyle{K}}
\newcommand{\sZ}{\scriptscriptstyle{Z}}
\newcommand{\sPi}{\scriptscriptstyle{\Pi}}
\newcommand{\sY}{\scriptscriptstyle{Y}}
\newcommand{\sscr}{\scriptscriptstyle}
\newcommand{\mms}{(X,\de_{{\scriptscriptstyle{X}}},\m_{{\scriptscriptstyle{X}}})}
\newcommand{\mmsi}{(X_i,\de_{{\scriptscriptstyle{X_i}}},\m_{{\scriptscriptstyle{X_i}}})}

\DeclareMathOperator{\frs}{\mathfrak{s}}
\DeclareMathOperator{\frc}{\mathfrak{c}}
\DeclareMathOperator{\vol}{vol}
\DeclareMathOperator{\Ch}{Ch}
\DeclareMathOperator{\lip}{Lip}

\DeclareMathOperator{\supp}{supp}

\DeclareMathOperator{\de}{d}
\DeclareMathOperator{\m}{m}
\DeclareMathOperator{\ric}{ric}

\DeclareMathOperator{\Ent}{Ent}

\newcommand{\ChX}{\Ch^{\sX}}

\title{Evolution variational inequality and Wasserstein control in variable curvature context}

\author{Christian Ketterer}
\email{christian.ketterer@math.uni-freiburg.de}

\begin{document}

\maketitle
\begin{abstract}In this note we continue the analysis of metric measure space with variable lower Ricci curvature bounds. First, we study $(\kappa,N)$-convex functions on metric spaces where $\kappa$ is a lower semi-continuous 
function, and
gradient flow curves in the sense of a new evolution variational inequality that captures the information that is provided by $\kappa$. Then, in the spirit of previous work by Erbar, Kuwada and Sturm \cite{erbarkuwadasturm}
we introduce an entropic curvature-dimension condition $CD^e(\kappa,N)$ for metric measure spaces and lower semi-continuous $\kappa$. This condition is stable with respect to Gromov convergence and we show that is equivalent to the 
reduced curvature-dimension condition $CD^*(\kappa,N)$ provided the space is essentially non-branching. Finally, we introduce 
a Riemannian curvature-dimension condition in terms of an evolution variational inequality on the Wasserstein space. A consequence is a new differential Wasserstein contraction estimate.
\end{abstract}
\noindent
%\keywords{
%keywords: }\\
%\keywords{
%AMS-codes: }
\tableofcontents
\section{Introduction}
In \cite{ketterer5} the author introduces a curvature-dimension condition 
$CD(\kappa,N)$ for a metric measure space $\mms$ in terms of displacement convexity on the $L^2$-Wasserstein space where $\kappa:X\rightarrow \mathbb{R}$ is a lower semi-continuous 
function and $N\in [1,\infty)$. If $\kappa$ is constant, the condition is precisely the definition
that was proposed by Lott, Sturm and Villani in \cite{stugeo1,stugeo2,lottvillani,viltot}. 
The condition $CD(\kappa,N)$
has geometric consequences as a generalized Bishop-Gromov estimates and a generalized Bonnet-Myers theorem. But it cannot recognize
Riemannian-type spaces which are characterized by linearity of the induced heat flow of their Cheeger energy. In the context of constant lower curvature bounds this obstacle was 
resolved by Ambrosio, Gigli and Savar\'e in \cite{agsheat,agsriemannian} who showed that displacement convexity in combination with linearity of the heat flow is equivalent to the existence of 
gradient flow curves in the sense of an \textit{evolution variational inequality} ($EVI$) for the Boltzmann Entropy on the $L^2$-Wasserstein. 

There are two important extensions of this idea. On the one hand, Erbar, Kuwada and Sturm \cite{erbarkuwadasturm} introduce a finite dimensional version of the $EVI$-formula 
 to define a Riemannian curvature-dimension condition.
As part of their program they also introduce
a so-called \textit{entropic curvature-dimension conditon} that is a more $PDE$-friendly modification of the original $CD$-condition. 
On the other hand, Sturm \cite{sturmvariable} defines $EVI_{\kappa}$-gradient flow curves where $\kappa$ is
a lower semi-continuous function. He proves several implications and equivalences, and also stability of this concept under measured Gromov convergence. 

In this note we introduce a Riemannian curvature-dimension condition for variable lower curvature bounds. For this purpose, we study $(\kappa,N)$-convex functions on metric spaces
where $\kappa:X\rightarrow \mathbb{R}$ is lower semi-continuous. We use a new characterization of $(\kappa,N)$-convexity in terms of an integrated inequality \cite{ketterer5} that involves 
so-called generalized dirstortion coefficients $\sigma_{\kappa}^{\sscr{(t)}}$ (Defintion \ref{generaldist}, Defintion \ref{ahr}). 
Provided the metric space $(X,\de_{\sX})$ admits a first variation formula, we can deduce an evolution variational inequality 
for gradient flow curves of $(\kappa,N)$-convex functions. 
More precisely, we say that an absolutely continuous curve $(x_s)_{s\in(0,\infty)}$ 
is an $EVI_{\kappa,\sN}$ gradient flow
curve of $f$ starting in $x_0\in X$ if $\lim_{s\rightarrow 0}x_s=x_0$, and if for all $z\in D(f)$ there exists a constant speed geodesic $\gamma^s:[0,1]\rightarrow X$ between 
$x_s$ and $z$ such that the \textit{evolution variational inequality} 
\begin{align}\label{jjj}
\frac{d}{dt}\sigma_{\kappa^-_{\gamma^s}/\sN}^{\sscr{(t)}}(|\dot{\gamma}^s|)|_{t=1}\geq \frac{1}{2N}\frac{d}{ds}\de_{\sX}(x_s,z)^2+\frac{d}{dt}\sigma_{\kappa^+_{\gamma^s}/\sN}^{\sscr{(t)}}(|\dot{\gamma}^s|)|_{t=0}\frac{U_N(z)}{U_N(x_s)}
\end{align}
holds for a.e. $s>0$ where $U_{\sN}(x)=e^{-\frac{f}{\sN}}$.
By monotonicity of the derivatives of the distortion coefficients at $0$ and $1$ the evolution variational inequality (\ref{jjj}) 
is consistent with the previous versions of evolution variational inequalities by Ambrosio, Gigli,
and Savar\'e and by Erbar, Kuwada and Sturm. Furthermore, an $EVI_{\kappa,\sN}$-gradient flow is also an 
$EVI_{\kappa}$-gradient flow in the sense of Sturm (Lemma \ref{monot}). In the special case of constant $\kappa$ our
$EVI$-inequality becomes the one in \cite{erbarkuwadasturm} (Remark \ref{thesame}).
In addition, we prove 
that the existence of $EVI_{\kappa,\sN}$-gradient curves yields strong $(\kappa,N)$-convexity (Theorem \ref{strong}).

Then, we use this idea in the context of the $L^2$-Wasserstein space and the Boltzmann Entropy over some metric measure space $\mms$. 
In the spirit of Erbar, Kuwada and Sturm we introduce an entropic curvature-dimension conditon $CD^e(\kappa,N)$. More precisely, 
a metric measure space $(X,\de_{\sX},\m_{\sX})$ satifies the \textit{entropic curvature-dimension condition} for some admissible function $\kappa$ and $N\geq 0$ if 
for any pair $\mu_0,\mu_1\in D(\Ent)$ with compact support there exists a $L^2$-Wasserstein geodesic $\Pi$ between $\mu_0$ and $\mu_1$ such that for all $t\in [0,1]$
\begin{align*}
U_N(\mu_t)\geq \sigma_{\kappa^-_{\Pi}/\sN\Theta^2}^{\sscr{(1-t)}}U_N(\mu_0)+\sigma_{\kappa^+_{\sPi}/\sN\Theta^2}^{\sscr{(t)}}U_N(\mu_1)
\end{align*}
where $U_{\sN}(\mu)=e^{-\frac{1}{\sN}\Ent}$, $(e_t)_{\star}\Pi=\mu_t$, $\Theta=W_2(\mu_0,\mu_1)$ and 
$$
\kappa_{\sPi}(t\Theta)\Theta^2=\int \kappa(e_t(\gamma)|\dot{\gamma}|^2d\Pi(\gamma).
$$
We show that the condition $CD^e$ is stable under measured Gromov convergence and that is equivalent to the reduced curvature-dimension condition $CD^*(\kappa,N)$ \cite{ketterer5} provided a non-branching assumption is satisfied.
Moreover, the entropic curvature-dimension condition already implies local compactness and finite Hausdorff dimension of the underlying metric measure space. 
%Let us point out that the entropic curvature-dimension condition is not just $(\kappa,N)$-convexity of the Entropy w.r.t. some lower semi-conitnuous function $\kappa$ on the $L^2$-Wasserstein space.

Then, we introduce a Riemannian curvature-dimension that is defined via combination of the entropic curvature-dimension condition and linearity of the heat flow. 
We show that the Riemannian curvature-dimension
condition can be characterized by the existence of \textit{Wasserstein $EVI_{\kappa,\sN}$-gradient flow curves} that is a straightforward modification of inequality (\ref{jjj}) in Wasserstein space context(Theorem \ref{B}). 
Furthermore, the Riemannian curvature-dimension condition is stable w.r.t. measured Gromov convergence (Theorem \ref{C}). 
Hence, these spaces arise naturally as non-smooth limit spaces of Riemannian manifolds (Corollary \ref{D}).
By monotonicity an $EVI_{\kappa,\sN}$ gradient flow curve is also an $EVI_{\kappa}$-gradient flow curve in the sense of Sturm \cite{sturmvariable}. The latter implies the following contraction estimate.
Let $\Pi^s$ be the unique $L^2$-Wasserstein geodesic between $\mu^s$ and $\nu^s$.
Then the following contraction estimate holds
\begin{align*}
\frac{d^+}{d s}W_2(\mu_s,\nu_s)^2
\leq -\int_0^1\int\kappa(\gamma(t))|\dot{\gamma}|^2d\Pi(\gamma)^sdt.
\end{align*}
We also show differential contraction estimates for $N<\infty$ that imply previous differential control estimates (Theorem \ref{dimensionalcontraction}).\smallskip

In section 2 we recall some important preliminaries on metric measure space and Wasserstein geometry. In section 3 introduce generalized distortion coefficients, $(\kappa,N)$-convexity on metric spaces, 
the evolution variational inequality, and prove several implications. In section 4 we introduce the entropic curvature-dimension condition for variable lower curvature bounds, and prove equivalence with the
reduced curvature-dimension condition in the context of essentially non-branching metric measure spaces. In section 5 we define Wasserstein-$EVI_{\kappa,\sN}$-gradient flow curves that characterize a Riemannian curvature 
dimension condition. In section 6 we deduce differential contraction estimates.
\section{Preliminaries}
\begin{definition}[Metric measure space]\label{assmms}
Let $(X,\de_{\sX})$ be a complete and separable metric space, 
and let $\m_{\sX}$ be a locally finite Borel measure on $(X,\de_{\sX})$.
That is,  for all $x\in X$ there exists $r>0$ such that $\m_{\sX}(B_r(x))\in(0,\infty)$.
Let $\mathcal{O}_{\sX}$ and $\mathcal{B}_{\sX}$ be the topology of open sets and the family of Borel sets, respectively.
A triple $(X,\de_{\sX},\m_{\sX})$ will be called \emph{metric measure space}. 
We assume that $\m_{\sX}(X)\neq 0$.
\end{definition}
$(X,\de_{\sX})$ is called \textit{length space} 
if $\de_{\sX}(x,y)=\inf \mbox{L} (\gamma)$ for all $x,y\in X$, 
where the infimum runs over all rectifiable curves $\gamma$ in $X$ connecting $x$ and $y$. 
$(X,\de_{\sX})$ is called \textit{geodesic space} 
if every two points $x,y\in X$ are connected by a curve $\gamma$ such that $\de_{\sX}(x,y)=\mbox{L}(\gamma)$.
Distance minimizing curves of constant speed are called \textit{geodesics}. 
A length space, which is complete and locally compact, is a geodesic space and proper (\cite[Theorem 2.5.23 ]{bbi}).
Rectifiable curves always admit a reparametrization proportional to arc length, and therefore become Lipschitz curves.
In general, we assume that a geodesic $\gamma:[0,1]\rightarrow X$ is parametrized proportional to its length. 
The set of all such geodesics $\gamma:[0,1]\rightarrow X$ is denoted with $\mathcal{G}(X)$ and 
the set of all Lipschitz curves $\gamma:[0,1]\rightarrow X$ parametrized proportional to arc-length is denoted with $\mathcal{LC}(X)$. 
$\mathcal{G}(X)$ and $\mathcal{LC}(X)$ are equipped with the topology that is induced by uniform convergence. 
More precisely, we always consider the distance $\de_{\infty}(\gamma,\tilde{\gamma})=\sup_{t\in[0,1]}|\gamma(t)-\tilde{\gamma}(t)|$.
\smallskip

Let $\mathcal{P}_2(X)$ be the $L^2$-Wasserstein space over $(X,\de_{\sX})$ equipped with the $L^2$-Wasserstein distance
$W_2$. 
The subspace of absolutely continuous probability measure with respect to $\m_{\sX}$ is denoted by $\mathcal{P}_2(\m_{\sX})$.
Recall that a \textit{dynamical optimal coupling} between $\mu_0,\mu_1\in \mathcal{P}_2(X)$ is a probability measure $\Pi$ 
on $\mathcal{P}(\mathcal{G}(X))$ such that $(e_0,e_1)_{\star}\Pi$ is an optimal coupling of $\mu_0$ and $\mu_1$. Then, the curve $t\in[0,1]\mapsto \mu_t=(e_t)_{\star}\Pi$ is
a geodesic in $\mathcal{P}_2(X)$ with respect to $W_2$. Moreover, for each geodesic $\mu_t$ in $\mathcal{P}_2(X)$ there exists a dynamical optimal plan $\Pi$. 
In rest of the article, 
we will not 
distinguish between $\mu_t$ and the corresponding probability measure $\Pi$ on $\mathcal{P}(X)$.
\section{$(\kappa,N)$-convex function and $EVI$ gradient flow curves}
\begin{theorem}[J. C. F. Sturm's comparison theorem]
Let $\kappa,\kappa':[a,b]\rightarrow \mathbb{R}$ be continuous function such that ${\kappa}'\geq {\kappa}\mbox{ on } [a,b]$ and $\frs_{\kappa'}>0$ on $(a,b]$. Then $\mathfrak{s}_{\kappa}\geq \mathfrak{s}_{{\kappa}'}$ on $[a,b]$.
\end{theorem}
\noindent
%\noindent
%\textit{Proof.} $\rightarrow $\cite[Chap. 10, Ex. 5]{carmo}
%\smallskip
%\\
A generalization of the previous theorem is the following result.
\begin{theorem}[Sturm-Picone oscillation theorem]
Let $\kappa,\kappa':[a,b]\rightarrow \mathbb{R}$ be continuous such that ${\kappa}'\geq {\kappa}\mbox{ on } [a,b]$. Let $u$ and $v$ be solutions of (\ref{ode}) with respect to $\kappa$ and $\kappa'$ respectively. If $u(a)=u(b)=0$
and $u>0$ on $(a,b)$, then either $u=\lambda v$ for some $\lambda>0$ or there exists $x_1\in (a,b)$ such that $v(x_1)=0$.
\end{theorem}
\begin{definition}[generalized $\sin$-functions]\label{sin}
Let $\kappa:[0,L]\rightarrow \mathbb{R}$ be a continuous function.
The generalized $\sin$ function $\frs_{\kappa}:[0,L]\rightarrow \mathbb{R}$
is the unique solution of
\begin{align}\label{ode}
v''+\kappa v=0.
\end{align} 
such that $\mathfrak{s}_{\kappa}(0)=0$ and $\mathfrak{s}_{\kappa}'(0)=1$.
The generalized $\cos$-function is $\mathfrak{c}_{\kappa}=\frs_{\kappa}'$.
\end{definition}

\begin{definition}[generalized distortion coefficients]\label{generaldist}
Consider $\kappa:[0,{L}]\rightarrow \mathbb{R}$ that is continuous and $\theta \in [0,L]$. 
Then
\begin{align*}
\sigma_{\kappa}^{\sscr{(t)}}(\theta)=\begin{cases}
\frac{\frs_{{\kappa}}(t\theta)}{\frs_{{\kappa}}(\theta)}& \mbox{ if }\frs_{\kappa}(t)>0 \mbox{ for all }t\in (0,\theta]\\
\infty & \mbox{ otherwise }.
\end{cases}
\end{align*}
If $\sigma_{\kappa}^{\sscr{(t)}}(\theta)<\infty$, $t\mapsto\sigma_{\kappa}^{\sscr{(t)}}(\theta)$ is a solution of 
\begin{align}\label{klebeband}
u''(t)+\kappa(t\theta)\theta^2u(t)=0
\end{align}
satisfying $u(0)=0$ and $u(1)=1$. We set $\sigma_{\kappa}^{\sscr{(t)}}(1)=\sigma_{\kappa}^{\sscr{(t)}}$. Then
$\sigma_{\kappa}^{\sscr{(t)}}(\theta)=\sigma_{\kappa\theta^2}^{\sscr{(t)}}$.
%We also define $$\pi_{\kappa}=\sup\left\{t\in[0,L]:\frs_{\kappa}(s)>0 \mbox{ for all }s\leq t\right\}.$$ 
\end{definition}
\noindent
If $\kappa:[0,L]\rightarrow \mathbb{R}$ is just lower semi-continuous, we can extend the previous definition in the following way. 
Define bounded continuous functions by
\begin{align}\label{uuu}
\kappa_n(x)=\min\left[\min_{y\in[0,L]}\left\{\kappa(y)+n|x-y|\right\},n\right]\ \ \ n\in\mathbb{N}.
\end{align}
The sequence $(\kappa_n)$ is montone increasing and converges pointwise to $\kappa$. Then, $\sigma_{\kappa_n}^{\sscr{(t)}}(\theta)$ is monotone increasing we define the generalized distortion coefficient with respect to $\kappa$ by
\begin{align*}
\sigma_{\kappa}^{\sscr{(t)}}(\theta)=\lim_{n\rightarrow\infty}\sigma_{\kappa_n}^{\sscr{(t)}}(\theta)\in[0,\infty].
\end{align*}
\begin{lemma}[\cite{ketterer5}]
Let $\kappa:[0,L]\rightarrow \mathbb{R}$ be lower semi-continuous and $\theta\in[0,L]$.
If $\sigma_{\kappa_{}}^{\sscr{(t_0)}}(\theta)=\infty$ for some $t_0\in (0,1)$ then $\sigma_{\kappa_{}}^{\sscr{(t)}}(\theta)=\infty$ for any $t\in (0,1)$.
\smallskip\\
In particular, either one has $\sigma_{\kappa_{}}^{\sscr{(t)}}(\theta)<\infty$ for any $t\in (0,1)$ and 
\begin{align*}
\sigma_{\kappa_{}}^{\sscr{(t)}}(\theta)={\frs_{\kappa_{}}(t\theta)}/{\frs_{\kappa}(\theta)}
\end{align*}
where $\frs_{\kappa_{}}(\theta)\neq 0$, or $\sigma_{\kappa_{}}^{\sscr{(\cdot)}}(\theta)\equiv\infty$. 
Here, $\frs_{\kappa}:[0,\theta]\rightarrow \mathbb{R}$ is the pointwise limit of the drecreasing sequence $\frs_{\kappa_n}:[0,\theta]\rightarrow \mathbb{R}$.
\end{lemma} 
\begin{remark}\label{someotherremark}
$\frs_{\kappa}$ is upper semi-continuous. Whence, $\sigma_{\kappa}^{\sscr{(t)}}(\theta)<\infty$ is upper semi-continuous in $t$. On the other hand, 
$\sigma_{\kappa}^{\sscr{(t)}}(\theta)$ is the limit of an non-decreasing sequence of continuous function. Therefore, $\sigma_{\kappa}^{\sscr{(t)}}(\theta)<\infty$ is continuous, 
and the sequence in (\ref{uuu}) converges uniformily.
The definition of $\sigma_{\kappa_{}}^{\sscr{(\cdot)}}(\theta)$ does not depend on the increasing sequence $\kappa_n$ that converges pointwise to $\kappa$.
\end{remark}
\begin{lemma}
If $\sigma^{\sscr{(t)}}_{\kappa}(\theta)<\infty$, we have
\begin{align*}
\sigma_{\kappa_{}}^{\sscr{(t)}}(\theta)=\int_0^1g(s,t)\theta^2\kappa\circ\gamma(s)\sigma_{\kappa_{}}^{\sscr{(s)}}(\theta)ds+t.
\end{align*}
with
$g(s,t)$ beeing the Green function of $[0,1]$. 
\end{lemma}
\begin{proof}
If $\kappa$ is continuous, this is clear.
If $\kappa$ is lower semi-continuous, we can choose $\kappa_n\uparrow \kappa$.
Then $\sigma^{\sscr{(t)}}_{\kappa_n}(\theta)\uparrow \sigma^{\sscr{(t)}}_{\kappa}(\theta)$ uniformily as $n\rightarrow \infty$ by Dini's theorem.
Then
\begin{align}\label{in}
\int_0^1\!g(s,t)\theta^2\kappa\circ\gamma(s)\sigma_{\kappa_{}}^{\sscr{(s)}}(\theta)ds&
=\lim_{M\rightarrow \infty}\int_0^1\!g(s,t)\theta^2\kappa\circ\gamma(s)\sigma_{\kappa_{}}^{\sscr{(s)}}(\theta)\wedge M ds\nonumber\\
&=\lim_{M\rightarrow \infty}\lim_{n\rightarrow\infty}\int_0^1\!g(s,t)\theta^2\kappa_n\circ\gamma(s)\sigma_{\kappa_{n}}^{\sscr{(s)}}(\theta)\wedge M ds\nonumber\\
&\leq\liminf_{n\rightarrow\infty}\int_0^1\!g(s,t)\theta^2\kappa_n\circ\gamma(s)\sigma_{\kappa_{n}}^{\sscr{(s)}}(\theta) ds=\sigma^{\sscr{(t)}}_{\kappa}(\theta)+t<\infty
\end{align}
where $t$ is fixed and $M>0$.
Hence, $g(s,t)\theta^2\kappa\circ\gamma(s)\sigma_{\kappa_{}}^{\sscr{(s)}}(\theta)$ is integrable in $s\in[0,1]$. If we apply the theorem of dominated convergence 
to $g(s,t)\theta^2\kappa_n\circ\gamma(s)\sigma_{\kappa_{n}}^{\sscr{(s)}}(\theta)$ we obtain equality in (\ref{in}).
\end{proof}
\begin{proposition}[\cite{ketterer5}]\label{monotonicity}
$\sigma^{\sscr{(t)}}_{\kappa}(\theta)$ is non-decreasing with respect to $\kappa:[0,\theta]\rightarrow \mathbb{R}$. More precisely
\begin{align*}
\kappa(x)\geq\kappa'(x) \ \forall x\in [0,\theta]\ \ \mbox{implies} \ \  \sigma^{\sscr{(t)}}_{\kappa}(\theta)\geq \sigma^{\sscr{(t)}}_{\kappa'}(\theta)\ \forall t\in [0,1].
\end{align*}
\end{proposition}
\begin{lemma}\label{uff}
Let $\kappa_i, \kappa:[0,L]\rightarrow \mathbb{R}$ be lower semi-continuous, such that
\begin{align*}
\liminf_{i\rightarrow\infty}\kappa_i(x)\geq \kappa(x)
\end{align*}
for any $x\in [0,L]$. Then
$
\liminf_{i\rightarrow \infty}\sigma_{\kappa_i}^{(t)}(\theta)\geq \sigma_{\kappa}^{(t)}(\theta)
$
for any $t\in[0,1]$.
\end{lemma}
\begin{proof} First, we assume that any distortion coefficient is finite.
Let $\epsilon>0$, and let $i_0\in\mathbb{N}$ such that for each $i\geq i_0$ 
\begin{align*}
\kappa_i(x)\geq \kappa(x)-\epsilon\geq \kappa_n-\epsilon
\end{align*}
for each $x\in[0,L]$ and $\kappa_n$ as before. By monotonicity, we have 
\begin{align*}
\liminf_{i\rightarrow\infty} \sigma_{\kappa_i}^{(t)}(\theta)\geq \sigma_{\kappa-\epsilon}^{(t)}(\theta)
\end{align*}
for each $t\in[0,1]$. $\kappa_n$ is a sequence of monotone increasing continuous functions converging to $\kappa$. 
Then, by stability of $\sigma_{\kappa}^{(t)}$ with respect to uniform change of $\kappa$, for each $\delta>0$ there exists $\epsilon>0$ such
that 
\begin{align*}
\sigma_{\kappa}^{(t)}(\theta)=\lim_{n\rightarrow\infty} \sigma_{\kappa_n}^{(t)}(\theta)\leq\lim_{n\rightarrow \infty}\sigma_{\kappa_n-\epsilon}^{(t)}(\theta)+\delta\leq \sigma_{\kappa-\epsilon}^{(t)}(\theta)+\delta\leq \liminf_{i\rightarrow \infty}\sigma_{\kappa_i}^{(t)}(\theta)+\delta
\end{align*}
where the first equality is the definition of $\sigma_{\kappa}^{(t)}(\theta)$ and the first inequality is continuity with respect to uniform changes of $\kappa_n$.
The only case that is left is when $\sigma_{\kappa}^{\sscr{(t)}}(\theta)=\infty$. But then by monotonicity it follows for a subsequence that 
$\kappa_i\uparrow \kappa$. From that we can deduce that $
\lim_{i\rightarrow \infty}\sigma_{\kappa_i}^{(t)}(\theta)=\infty
$ for any $t\in [0,1]$ by comparison the coefficients w.r.t. $\kappa_i$ and with the coefficients w.r.t $\kappa_n\uparrow \kappa$. 
\end{proof}
\noindent 
We define $\kappa^-:[0,\theta]\rightarrow\mathbb{R}$ by $\kappa^-(x)=\kappa(\theta-x)$.
\begin{lemma}[\cite{ketterer5}]\label{gr}
Consider $\kappa,\kappa':[0,\theta]\rightarrow \mathbb{R}$. Then
\begin{align*}
\sigma^{\sscr (t)}_{\kappa_{}}(\theta)^{1-\lambda}\cdot\sigma^{\sscr (t)}_{\kappa_{}'}(\theta)^{\lambda}\geq \sigma^{\sscr (t)}_{(1-\lambda)\kappa_{}+\lambda\kappa_{}'}(\theta).
\end{align*}
Especially, $\kappa\in LSC([0,1])\mapsto \log\sigma_{\kappa_{}}$ and $\kappa\mapsto \log\sigma_{\kappa^-_{}}$ are convex. $LSC([0,1])$ denotes the space of lower semi-continuous functions.
\end{lemma}
\begin{corollary}\label{grrr}
For $t\in[0,1]$ the function $G:\mathbb{R}^2\times {LSC}([0,1])\rightarrow \mathbb{R}\cup \left\{\infty\right\}$ given by
\begin{align*}
G(x,y,\kappa)=\log \left[\sigma_{\kappa^-_{}}^{(1-t)}e^x+\sigma_{\kappa^+_{}}^{(t)}e^y\right]\ \mbox{ is convex.}
\end{align*}
\end{corollary}
\begin{proof} We argue like in \cite[Lemma 2.11]{erbarkuwadasturm}.
Note that $$G(x,y,\kappa)=F(\log \sigma_{\kappa^-_{}}^{(1-t)}+x,\log \sigma_{\kappa^+_{}}^{(t)}+y)$$ with $F(u,v)=\log(e^u+e^v)$. Then 
the assertion follows since $F$ is convex, $x\mapsto F(u+x,v+x)$ is monotone increasing and $\kappa\mapsto \log \sigma_{\kappa_{}^{+/-}}$ is convex.
\end{proof}
\begin{remark}\label{ody}
If $\Pi\in\mathcal{P}(LSC([0,1])$, then $\kappa_{\Pi}:t\mapsto \int\kappa(t)d\Pi(\kappa)\in LSC([0,1])$ by Fatou's Lemma.
Hence, if $f:LSC([0,1])\rightarrow \mathbb{R}$ is convex, we obtain $$\int f(\kappa)d\Pi(\kappa)\leq f(\kappa_{\Pi}).$$
\end{remark}

For the rest of the article we use the following notation. Let $(X,\de_{\sX})$ be a metric space, and let $\kappa:X\rightarrow\mathbb{R}$ be lower semi-continuous.
We set $\kappa_{\gamma}=\kappa\circ\bar{\gamma}$ where $\gamma:[0,1]\rightarrow X$ is a constant speed geodesic in $X$ and $\bar{\gamma}$ its unit speed reparametrization.
We denote with $\gamma^-(t)=\gamma(1-t)$ the reverse parametrization of $\gamma$, and we also write $\gamma=\gamma^{\sscr{+}}$, and $\kappa^{\sscr{-/+}}_{\gamma}:=\kappa_{\gamma^{-/+}}$.
\begin{proposition}[\cite{ketterer5}]\label{central}
Let $\kappa:[a,b]\rightarrow \mathbb{R}$ be lower semi-continuous and $u:[a,b]\rightarrow \mathbb{R}_{\geq 0}$ be upper semi-continous. 
Then the following statements are equivalent:
\begin{itemize}
 \item[(i)]$u''+\kappa u\leq 0$ in the distributional sense, that is 
\begin{align}\label{distributional}
\int_a^b\varphi''(t)u(t)dt\leq -\int_a^b\varphi(t)\kappa(t)u(t)dt
\end{align}
for any $\varphi\in C_0^{\infty}\left((a,b)\right)$ with $\varphi\geq 0$.
\smallskip
 \item[(ii)] It holds 
\begin{align}
u(\gamma(t))\geq (1-t)u(\gamma(0))+tu(\gamma(1)+\int_0^1g(t,s)\kappa(\gamma(s))\theta^2 u(\gamma(s)) ds
\end{align}
for any constant speed geodesic $\gamma:[0,1]\rightarrow [a,b]$ where $\theta=|\dot{\gamma}|=\mbox{L}(\gamma)$ with
$g(s,t)$ beeing the Green function of $[0,1]$. 
\smallskip
 \item[(iii)] There is a constant $0<L\leq b-a$ such that
 \begin{align}\label{kuconcavity}
u(\gamma(t))\geq \sigma^{\sscr{(1-t)}}_{\kappa^{\sscr{-}}_{\gamma}}(\theta)u(\gamma(0))+\sigma^{\sscr{(t)}}_{\kappa^{\sscr{+}}_{\gamma}}(\theta)u(\gamma(1))
 \end{align}
for any constant speed geodesic $\gamma:[0,1]\rightarrow [a,b]$ with 
$\theta=|\dot{\gamma}|=\mbox{L}(\gamma)\leq L$.  We use the convention $\infty\cdot 0=0$.
\smallskip
 \item[(iv)] (\ref{kuconcavity}) holds for any constant speed geodesic $\gamma:[0,1]\rightarrow [a,b]$.
\end{itemize}
\end{proposition}
\begin{definition}\label{ahr}
Consider a metric space $(Y,\de_{\sY})$ and a continuous function $\kappa:Y\rightarrow \mathbb{R}$. 
We say a function $u:Y\rightarrow [0,\infty)$ is \textit{$\kappa u$-concave} if 
for each geodesic $\gamma:[0,1]\rightarrow D(u)=\left\{u>0\right\}$
\begin{align}\label{grugru}
u(\gamma(t))\geq \sigma^{\sscr{(1-t)}}_{\kappa^{\sscr{-}}_{\gamma}}(\mbox{L}(\gamma))u(\gamma(0))+\sigma^{\sscr{(t)}}_{\kappa^{\sscr{+}}_{\gamma}}(\mbox{L}(\gamma))u(\gamma(1))
\end{align}
where $\kappa_{\gamma}=\kappa\circ\bar{\gamma}:[0,\mbox{L}(\gamma)]\rightarrow Y$, $\bar{\gamma}$ is the unit speed reparametrization of $\gamma$ and $D(u)$ is equipped with the induced metric. 
\smallskip\\
We say $u:Y\rightarrow [0,\infty)$ is \textit{weakly $\kappa u$-concave} if for all $x,y\in D(u)$ there exists a geodesic $\gamma:[0,1]\rightarrow Y$ between $x$ and $y$ such that (\ref{grugru}) holds.
To distinguish $\kappa u$-concavity from weak $\kappa u$-concavity we will
also say \textit{strong $\kappa u$-convexity}.
\smallskip\\
%\item[(ii)] 
Consider a function $S:Y\rightarrow \mathbb{R}\cup\left\{\infty\right\}$. We say $S$ is (weakly) $(\kappa,N)$-convex 
if $e^{-\frac{S}{N}}=:U_{\sN}: Y\rightarrow [0,\infty)$ is (weakly) $\frac{\kappa}{N} U_{\sN}$-concave.
We use the convention $e^{-\infty}=0$. Again, we will also say that $S$ is \textit{strongly $(\kappa,N)$-convex} instead of $S$ is $(\kappa,N)$-convex.
\end{definition}
\begin{lemma}\label{sums} Let $Y$ be a metric space as in the previous definition.
\begin{itemize}
 \item[(i)] If $S$ is (weakly) $(\kappa,N)$-convex, then $\lambda\cdot f$ is (weakly) $(\lambda\kappa,\lambda N)$-convex for $\lambda>0$.
 \smallskip
 \item[(ii)] If $S_1$ is weakly $(\kappa_1,N_1)$-convex and $S_2$ is strongly $(\kappa_2,N_2)$-convex, then $S_1+S_2:D(S_1)\cap D(S_2)\rightarrow [-\infty,\infty)$ is weakly $(\kappa_1+\kappa_2,N_1+N_2)$-convex.
 \smallskip
 \item[(iii)] If $S$ is (weakly) $(\kappa, N)$-convex and $\kappa'\leq \kappa$ and $N'\geq N$ then $S$ is (weakly) $(\kappa',N')$-convex. 
\end{itemize}
\begin{proof}
Apply Proposition \ref{monotonicity} and Lemma \ref{gr}. 
\end{proof}

\end{lemma}
\begin{corollary}
If $S:Y\rightarrow (-\infty,\infty]$ is a (weakly) $(\kappa,N)$-convex function then $S$ is (weakly) $\kappa$-convex in the sense that for each geodesic $\gamma$ in $D(S)$
(for each pair $x_0,x_1\in D(S)$ there exists 
a geodesic $\gamma:[0,1]\rightarrow Y$ with)
\begin{align*}
S(\gamma_t)\leq (1-t)S(x_0)+tS(x_1)- \int_0^1g(s,t)\de_{\sY}(x_0,x_1)^2\kappa\circ{\gamma}(s)dt
\end{align*} for all $t\in [0,1]$.
\end{corollary}
\begin{proof}
Let us assume that $S$ is $(\kappa,N)$-convex. 
Consider a geodesic $\gamma$. Then (\ref{distributional}) holds for $u(s)=e^{-\frac{1}{N}S\circ \gamma(s)}$ for $\kappa/N$ instead of $\kappa$. Multiply with $N$ and let $N\rightarrow \infty$. 
We obtain a distributional inequality for $S$ along $\gamma$ that characterizes $\kappa$-convexity.
\end{proof}

\begin{remark}
Let $(M,g_{\sM})$ be a Riemannian manifold. A smooth function $u:M\rightarrow [0,\infty)$ is (weakly) $\kappa u$-convex if 
$
\nabla^2 u\geq \kappa u g_{\sM}
$
and a smooth function $f:M\rightarrow \mathbb{R}$ is $(\kappa,N)$-convex if 
$
\nabla^2 f\geq \kappa g_{\sM}+\frac{1}{N}\left(df\otimes df\right).
$
\end{remark}

\noindent 
For a function $f:[a,b]\rightarrow \mathbb{R}$ the right and left derivatives are denoted by 
\begin{align*}
\frac{d^+}{dt}f(t)=\limsup_{h\downarrow 0}\frac{f(t+h)-f(t)}{h}\ \ \&\ \ 
\frac{d^-}{dt}f(t)=\liminf_{h\uparrow 0}\frac{f(t+h)-f(t)}{h}.
\end{align*}
We can consider the derivative of $\sigma^{\sscr{(t)/(1-t)}}_{\kappa}(\theta)$ in the following sense:
\begin{align*}
\begin{cases}
\frac{d^+}{dt}\Big|_{t=0}\sigma_{\kappa_{}}^{\sscr{(t)/(1-t)}}(\theta)\in \mathbb{R}&\ \mbox{ if }\sigma^{\sscr{(t)/(1-t)}}_{\kappa}(\theta)<\infty\\
\infty & \ \mbox{ otherwise}.
\end{cases}
\end{align*}
\begin{lemma}\label{derivatives} Consider $\kappa\in LSC([0,L])$ and $\theta\in[0,L]$ such that $\sigma^{\sscr{(t)/(1-t)}}_{\kappa}(\theta)<\infty$. Then
\begin{align}\label{advancedmono}
\frac{d^+}{dt}\Big|_{t=0}\sigma_{\kappa_{}}^{\sscr{(t)}}(\theta)\uparrow \ \mbox{ and } \ 
\frac{d^-}{dt}\Big|_{t=1}\sigma_{\kappa_{}}^{\sscr{(t)}}(\theta)=-\frac{d^+}{dt}\Big|_{t=0}\sigma_{\kappa_{}}^{\sscr{(1-t)}}(\theta)\downarrow 
\end{align}
if $\kappa$ is non-decreasing.
\end{lemma}
\begin{proof}
The distortion coefficients $\sigma_{\kappa_{}}^{\sscr{(t)}}(|\dot{\gamma}|)$ are monotone increasing with respect to $\kappa$ and satisfy 
$\sigma_{\kappa_{\gamma}}^{\sscr{(0)}}(|\dot{\gamma}|)=0$ and $\sigma_{\kappa_{\gamma}}^{\sscr{(1)}}(|\dot{\gamma}|)=1$.
\end{proof}
\begin{remark}
If $\sigma^{\sscr{(t)}}_{\kappa}(\theta)<\infty$,
$
\frac{d^+}{dt}|_{t=0}\sigma_{\kappa_{}}^{\sscr{(t)}}(\theta)$ has the following respresentation.
\begin{align*}
\frac{d^+}{dt}\Big|_{t=0}\sigma_{\kappa_{}}^{\sscr{(t)}}(\theta)=\int_0^1(1-s)\theta^2\kappa\circ\gamma(s)\sigma_{\kappa_{}}^{\sscr{(s)}}(\theta)ds+1.
\end{align*}
This follows from 
\begin{align*}
\sigma_{\kappa_{}}^{\sscr{(t)}}(\theta)-t=\int_0^1g(s,t)\theta^2\kappa\circ\gamma(s)\sigma_{\kappa_{}}^{\sscr{(s)}}(\theta)ds
\end{align*}
and $\frac{d}{dt}|_{t=0}g(s,t)=(1-s)$. Similar for $\frac{d^-}{dt}|_{t=1}\sigma_{\kappa_{}}^{\sscr{(t)}}(\theta)$
\end{remark}

\begin{lemma}\label{uniform}
Let $\kappa_n,\kappa:[0,\theta]\rightarrow\mathbb{R}$ be lower semi-continuous such that $$\liminf_{n\rightarrow\infty}\kappa_n(x)\geq \kappa(x).$$ 
Assume that $\frac{d^+}{dt}|_{t=0}\sigma_{\kappa_{n}}^{\sscr{(t)}}(\theta), 
\frac{d^+}{dt}|_{t=0}\sigma_{\kappa_{}}^{\sscr{(t)}}(\theta)<\infty$. Then
\begin{align}
\liminf_{n\rightarrow\infty}\frac{d^+}{dt}\Big|_{t=0}\sigma_{\kappa_{n}}^{\sscr{(t)}}(\theta)\geq 
\frac{d^+}{dt}\Big|_{t=0}\sigma_{\kappa_{}}^{\sscr{(t)}}(\theta).
\end{align}
\end{lemma}
\begin{proof}
If $\kappa$ is continuous, then uniform changes of $\kappa$ imply that $\frac{d}{dt}|_{t=0,1}\sigma_{\kappa_{}}^{\sscr{(t)}}(\theta)$ changes uniformily by the previous remark.
The rest of the proof works as in Lemma \ref{uff}.
\end{proof}
\begin{lemma}\label{summerday}
Let $\kappa$ be continuous and let $u\in C^1([a,b])$. Then $u$ satisfies one of the equivalent statements in Proposition \ref{central} 
if and only if there exists a constant $L\in (0,b-a)$ such that
\begin{align}\label{9}
\frac{d}{dt}\sigma_{\kappa_{\gamma}}^{\sscr{(t)}}(|\dot{\gamma}|)|_{t=0}u\circ\gamma(1)\leq \frac{d}{dt}\sigma_{\kappa_{\gamma}}^{\sscr{(t)}}(|\dot{\gamma}|)|_{t=1}u\circ\gamma(0)+(u\circ\gamma)'(0)
\end{align}
%or equivalently
%\begin{align*}
%u\circ\gamma(1)\leq {\frc_{\kappa^-}(|\dot{\gamma}|)}\frac{\frs_{\kappa^{+}}(|\dot{\gamma}|)}{\frs_{\kappa^{-}}(|\dot{\gamma}|)}u\circ\gamma(0)+\frac{\frs_{\kappa^{+}}(|\dot{\gamma}|)}{|\dot{\gamma}|}(u\circ\gamma)'(0)
%\end{align*}
for any constant speed geodesic $\gamma:[0,1]\rightarrow [a,b]$. 
\end{lemma}
\begin{proof}
``$\Rightarrow$'': Consider (\ref{kuconcavity}). We add $u\circ \gamma(0)$ on both sides of the inequality and devide by $t$.
\begin{align*}
\frac{1}{t}\left(u(\gamma(t))-u(\gamma(0))\right)\geq \frac{1}{t}\left(\sigma_{\kappa^-_{\gamma}}^{\sscr{(1-t)}}(|\dot{\gamma}|)-1\right)u(\gamma(0))+\frac{1}{t}{\sigma_{\kappa^+_{\gamma}}^{\sscr{(t)}}(|\dot{\gamma}|)}u(\gamma(1))
\end{align*}
Taking the limit $t\rightarrow 0$ yields
\begin{align}\label{8}
(u\circ\gamma)'(0)\geq -\frac{d}{dt}\sigma_{\kappa^-_{\gamma}}^{\sscr{(t)}}(|\dot{\gamma}|)|_{t=1}u(\gamma(0))+\frac{d}{dt}\sigma_{\kappa^+_{\gamma}}^{\sscr{(t)}}(|\dot{\gamma}|)|_{t=0}u(\gamma(1)).
\end{align}
\smallskip\\
``$\Leftarrow$'':
If $\kappa_{\gamma}\geq K\in\mathbb{R}$, Lemma \ref{derivatives} implies
\begin{align*}
\frac{d}{dt}|_{t=0}\sigma_{\kappa_{\gamma}}^{\sscr{(t)}}(|\dot{\gamma}|)\geq\frac{d}{dt}|_{t=0}\sigma_{K}^{\sscr{(t)}}(|\dot{\gamma}|).
\end{align*}
and similar for $\sigma_{\kappa_{\gamma}^-}^{\sscr{(1-t)}}(|\dot{\gamma}|)$. Hence
\begin{align*}
\frac{\frs_{K}(|\dot{\gamma}|)}{|\dot{\gamma}|}(u\circ\gamma)'(0)\geq -\frc_K(|\dot{\gamma}|)u(\gamma(0))+u(\gamma(1))
\end{align*}
Now, we pick a point $r\in [a,b]$. Then, for each $\epsilon>0$ one can pick a geodesic $\gamma:[0,1]\rightarrow [a,b]$ such that $\gamma_{\frac{1}{2}}=r$ and $|\dot{\gamma}|=\epsilon$.
If we set $\min \kappa_{\gamma}=K_{\epsilon}$, we can deduce exactly like in Lemma 2.2 in \cite{erbarkuwadasturm} that 
\begin{align*}
(u\circ\bar{\gamma})''\leq - K_{\epsilon}u\circ\bar{\gamma}
\end{align*}
pointwise on $[0,\epsilon]$.
Since $K_{\epsilon}\rightarrow \kappa (r)$ for $\epsilon \rightarrow 0$ and $\mbox{Im}(\gamma)\rightarrow r$, the result follows.
\end{proof}
\begin{lemma}\label{alemma}
Let $f$ be a smooth $(\kappa,N)$-convex function on a Riemannian manifold $(M,g)$, and let $U_N$ be as in Definition \ref{ahr}. 
Then, a smooth curve $x:[0,\infty)\rightarrow M$ is a gradient flow curve of $f$ if and only if
for each $z\in M$ and all $t>0$ we have that
\begin{align*}
-\frac{1}{2N}\frac{d}{ds}\de_{\sM}(x_s,z)^2+\frac{d}{dt}\sigma_{\kappa^-_{\gamma^s}/\sN}^{\sscr{(t)}}(|\dot{\gamma}^s|)|_{t=1}\geq \frac{d}{dt}\sigma_{\kappa^+_{\gamma^s}/\sN}^{\sscr{(t)}}(|\dot{\gamma}^s|)|_{t=0}\frac{U_N(z)}{U_N(x_s)}.
\end{align*}
where $\gamma^s:[0,1]\rightarrow M$ is the constant speed geodesic between $x_s$ and $z$.
\end{lemma}
\begin{remark}\label{thesame}
In the case of constant curvature $\kappa$ the inequality becomes
\begin{align*}
\cos_{\kappa/\sN}(\de_{\sM}(x_s,z))-\frac{1}{2N}\sin_{\kappa/\sN}(\de_{\sM}(x_s,z))\frac{d}{ds}\de_{\sM}(x_s,z)\geq \frac{U_N(z)}{U_N(x_s)}
\end{align*}
If $\kappa>0$, we can write
\begin{align*}
\frc_{\kappa/\sN}(\de_{\sM}(x_s,z))+\frac{1}{2\kappa}\frac{d}{ds}\frc_{\kappa/\sN}(\de_{\sM}(x_s,z))\geq \frac{U_N(z)}{U_N(x_s)}.
\end{align*}
And similar for $\kappa<0$.
Using the transformation $$\frac{1}{N}\frs_{\kappa/\sN}(x/2)^2=\frac{1}{2\kappa}(1-\frc_{\kappa/\sN}(x))$$ this becomes the $\mbox{evi}_{\kappa,\sN}$ formula from \cite{erbarkuwadasturm}.
\end{remark}
\begin{proof}
``$\Rightarrow$'': Let $x_t$ be a gradient flow curve of $f$. Then, by the first variation formula we can compute
\begin{align*}
\frac{d}{dt}\Big|_{t=0} U_N({\gamma^s_t})=-\frac{1}{N}U_{\sN}(x_t)g_{\sM}(\nabla f|_{x_t}, \dot{\gamma}_0)=-\frac{1}{2N}U_{\sN}(x_t)\frac{d}{dt}\de_{\sM}(x_t,z)^2.
\end{align*}
with Lemma \ref{summerday} the result follows immediately.
\smallskip\\
``$\Leftarrow$'': Since $\kappa$ is bounded from below, the backward direction follows from monotonicity of the distortion coefficients and their derivatives like in Lemma 2.4 in \cite{erbarkuwadasturm}.
\end{proof}
\paragraph{\textbf{Evolution variational inequality in metric spaces}}
Let $(X,\de_{\sX})$ be a complete and separable metric space, and let $f:X\rightarrow (-\infty,\infty]$ be a lower semi-continuous function. 
We repeat some definitions from differential calculus
on metric spaces. The descending slop of $f$ at $x$ is 
\begin{align*}
|\nabla^-f|(x):=\limsup_{y\rightarrow x}\frac{[f(x)-f(y)]_+}{\de_{\sX}(x,y)}.
\end{align*}
A curve $x:[a,b]\rightarrow X$ is called absolutely continuous if 
\begin{align}\label{curves}
\de_{\sX}(x_t,x_s)\leq \int_t^sg(r)dr \ \ \mbox{ for all }t,s\in [a,b]\mbox{ such that }s\leq t
\end{align}
and some $g\in L^1([a,b])$. We say $x$ is locally absolutely continuous if (\ref{curves}) holds locally in $[a,b]$. 
For an absolutely continuous curve $x$ the metric speed 
\begin{align*}
|\dot{x}|(t):=\lim_{h\rightarrow 0}\frac{\de_{\sX}(x_{t+h},x_t)}{|h|}
\end{align*}
exists for a.e. $t\in[a,b]$ and is the minimal $g$ in (\ref{curves}).
\begin{definition}
A locally absolutely continuous curve $x:[0,\infty)\rightarrow X$ with $x(0)\in X$ is a gradient flow curve of $f$ starting in $x(0)$ if the \textit{energy dissipation 
equality}
\begin{align}\label{dissipation}
 f(x(s))=f(c(t))+\frac{1}{2}\int_s^t\left(|\dot{x}|^2(r)+|\nabla^-f|(x(r))\right)dr \mbox{ for all }0\leq s \leq t
\end{align}
holds.
\end{definition}
\noindent
Lemma \ref{alemma} motivates the following definition.
\begin{definition} Let $f$ be as before.
Let $\kappa:X\rightarrow \mathbb{R}$ be a lower semi-continuous function, $N\geq 1$ and let $x:(0,\infty)\rightarrow D(f)$ be a 
locally absolutely continuous curve. We say that $x_s$ is an $EVI_{\kappa,\sN}$ gradient flow
curve of $f$ starting in $x_0\in X$ if $\lim_{s\rightarrow 0}x_s=x_0$, and if for all $z\in D(f)$ there exists a constant speed geodesic $\gamma^s:[0,1]\rightarrow X$ between 
$x_s$ and $z$ such that $$-\frac{d}{dt}\sigma_{\kappa^-_{\gamma^s}/\sN}^{\sscr{(t)}}(|\dot{\gamma}^s|)|_{t=1}<\infty \ \ \& \ \ \frac{d}{dt}\sigma_{\kappa^+_{\gamma^s}/\sN}^{\sscr{(t)}}(|\dot{\gamma}^s|)|_{t=0}<\infty$$ and 
the \textit{evolution variational inequality} 
\begin{align*}
-\frac{1}{2N}\frac{d}{ds}\de_{\sX}(x_s,z)^2+\frac{d}{dt}\sigma_{\kappa^-_{\gamma^s}/\sN}^{\sscr{(t)}}(|\dot{\gamma}^s|)|_{t=1}\geq \frac{d}{dt}\sigma_{\kappa^+_{\gamma^s}/\sN}^{\sscr{(t)}}(|\dot{\gamma}^s|)|_{t=0}\frac{U_N(z)}{U_N(x_s)}
\end{align*}
holds for a.e. $s>0$. If $N=\infty$, we say $x_s$ is an $EVI_{\kappa,\infty}$ gradient flow curve of $f$ if for all $z\in D(f)$ there exists a constant speed geodesic $\gamma^s:[0,1]\rightarrow X$ between 
$x_s$ and $z$ such that
\begin{align*}
\frac{d}{ds}\frac{1}{2}\de_{\sX}(x_s,z)^2+\int_0^1\kappa(\gamma^s(t))dt\de_{\sX}(x_s,z)^2\leq f(x_s)-f(z)
\end{align*}
holds for a.e. $s>0$.
\end{definition}
\begin{remark}
The definition of $EVI_{\kappa,\infty}$ for gradient flows already appears \cite{sturmvariable}.
\end{remark}

\begin{lemma}\label{monot}
If $(x_s)_{s\in[0,\infty)}$ is an $\mbox{evi}_{\kappa,\sN}$ gradient flow curve of $f$, then it is also an $\mbox{evi}_{\kappa',\sN'}$ gradient flow curve for any $\kappa'\leq \kappa$ and
$N'\geq N$ where $\kappa'$ is a lower semi-continuous function and $N'\in[N,\infty]$. 
\end{lemma}
\begin{proof}
The case $N'<\infty$ follows directly from Lemma \ref{derivatives}. If $N'=\infty$, we rewrite the $EVI_{\kappa,\sN}$ formula as
\begin{align*}
&-\frac{1}{2N}\frac{d}{ds}\de_{\sX}(x_s,z)^2-\frac{d}{dt}\left[\sigma_{\kappa^-_{\gamma^s}/\sN}^{\sscr{(1-t)}}(|\dot{\gamma}^s|)+\sigma_{\kappa^+_{\gamma^s}/\sN}^{\sscr{(t)}}(|\dot{\gamma}^s|)\right]_{t=0}\\
&\hspace{6cm}\geq \frac{d}{dt}\sigma_{\kappa^+_{\gamma^s}/\sN}^{\sscr{(t)}}(|\dot{\gamma}^s|)|_{t=0}\left[\frac{U_N(z)}{U_N(x_s)}-1\right]
\end{align*}
If we multiply the inequality with $N$, we see that the right hand side converges to $-f(z)+f(x_s)$ for $N\rightarrow \infty$.
To see what happens on the left hand side we define for fixed $s\in (0,\infty)$
$$
\phi(t):=\sigma_{\kappa^-_{\gamma^s}/\sN}^{\sscr{(1-t)}}(|\dot{\gamma}^s|)+\sigma_{\kappa^+_{\gamma^s}/\sN}^{\sscr{(t)}}(|\dot{\gamma}^s|)-1.
$$
$\phi$ solves $u''(t)+\kappa(\gamma(t))/N \de_{\sX}(z,x_s)\left[\phi(t)+1\right]=0$ with $u(0)=u(1)=0$. Hence, we have the following respresentation 
\begin{align*}
\phi(t)=\int_0^1g(\tau,t)\kappa(\gamma^s(\tau))/N \de_{\sX}(z,x_s)\left[\sigma_{\kappa^-_{\gamma^s}/\sN}^{\sscr{(1-\tau)}}(|\dot{\gamma}^s|)+\sigma_{\kappa^+_{\gamma^s}/\sN}^{\sscr{(\tau)}}(|\dot{\gamma}^s|)\right]d\tau
\end{align*}
and 
\begin{align*}
\phi'(0)
%&=\frac{d}{dt}\left[\sigma_{\kappa^-_{\gamma^s}/\sN}^{\sscr{(1-t)}}(|\dot{\gamma}^s|)+\sigma_{\kappa^+_{\gamma^s}/\sN}^{\sscr{(t)}}(|\dot{\gamma}^s|)\right]_{t=0}\\
&=\int_0^1(1-\tau) \kappa(\gamma^s(\tau))/N \de_{\sX}(z,x_s)\left[\sigma_{\kappa^-_{\gamma^s}/\sN}^{\sscr{(1-\tau)}}(|\dot{\gamma}^s|)+\sigma_{\kappa^+_{\gamma^s}/\sN}^{\sscr{(\tau)}}(|\dot{\gamma}^s|)\right]d\tau.
\end{align*}
If $N\rightarrow \infty$, we see that $\phi$ converges uniformily to $0$. Therefore, by the previous formula of $\phi'(0)$, $N\phi'(0)$ converges to 
$
\int_0^1(1-\tau)\kappa(\gamma^s(\tau))\de_{\sX}(z,x_s)d\tau.
$
\end{proof}
\begin{theorem}\label{strong}
Let $(X,\de_{\sX})$ be a locally compact metric measure space, and let $f:D(f)\rightarrow \mathbb{R}$ be lower semi-continuous.
Assume that for every $x_0\in \overline{D(f)}$ there exists an $EVI_{\kappa,\sN}$ gradient flow curve $(x_s)_{s\in(0,\infty)}$ starting in $x_0$. Then $f$ is strongly $(\kappa,N)$-convex.
\end{theorem}
\begin{proof}
First, we assume that $\kappa:X\rightarrow \mathbb{R}$ is continuous.
Let ${c}:[0,1]\rightarrow X$ be a constant speed geodesic, and let $\bar{c}:[0,\theta]\rightarrow X$ its 1-speed reparametrization.
Let $\delta>0$ be arbitrary.
Since $(X,\de_{\sX})$ is locally compact, we can find ${h}>0$ and points $r_i\in[0,\theta]$ for $i=1,\dots,N$ such that
\begin{align*}
\max\kappa|_{B_{2{h}}(\bar{c}(r_i))}-\min\kappa|_{B_{2{h}}(\bar{c}(r_i))}<\delta
\end{align*}
for each $i=1,\dots,N$. 
Now, we pick $\hat{r}\in [0,\theta]$ and $\epsilon>0$, and consider $\bar{\gamma}=\bar{c}|_{[\hat{r}-\epsilon,\hat{r}+\epsilon]}$ 
such that $\hat{r}\pm\epsilon\in[r_i-{h},r_i+{h}]$ for some $i=1,\dots,N$.
Its constant speed reparametrization is $\gamma_{}:[0,1]\rightarrow X$.
Let $x_s$ be the $EVI_{\kappa,\sN}$ gradient flow curve starting in $\gamma(\textstyle{\frac{1}{2}})$. Then, we obtain
\begin{align*}
-\frac{1}{2N}\frac{d}{ds}\de_{\sX}(x_s,\gamma_{}(0))^2+\frac{d}{dt}\sigma_{\kappa^-_{\gamma_0^s}/\sN}^{\sscr{(t)}}(|\dot{\gamma}_0^s|)|_{t=1}\geq \frac{d}{dt}\sigma_{\kappa^+_{\gamma_0^s}/\sN}^{\sscr{(t)}}(|\dot{\gamma}_0^s|)|_{t=0}\frac{U_N(\gamma_{}(0))}{U_N(x_s)}
\end{align*}
and
\begin{align*}
-\frac{1}{2N}\frac{d}{ds}\de_{\sX}(x_s,\gamma_{}(1))^2+\frac{d}{dt}\sigma_{\kappa^-_{\gamma_1^s}/\sN}^{\sscr{(t)}}(|\dot{\gamma}_1^s|)|_{t=1}\geq \frac{d}{dt}\sigma_{\kappa^+_{\gamma_1^s}/\sN}^{\sscr{(t)}}(|\dot{\gamma}_1^s|)|_{t=0}\frac{U_N(\gamma_{}(1))}{U_N(x_s)}
\end{align*}
where $\gamma^s_i:[0,1]\rightarrow X$ $i=0,1$ are constant speed geodesics between $x_s$ and $\gamma(i)$.
Since $\gamma_{}$ is a constant speed geodesic, it follows
\begin{align*}
\frac{1}{2}\de_{\sX}(\gamma_{}(0),x_s)^2+\frac{1}{2}\de_{\sX}(\gamma_{}(1),x_s)^2\geq \frac{1}{2}\de_{\sX}(\gamma(0),x_0)^2+\frac{1}{2}\de_{\sX}(\gamma(1),x_0)^2
\end{align*}
and therefore
\begin{align*}
\frac{1}{2}\frac{d}{ds}\Big|_{s=0}\de_{\sX}(\gamma_{}(0),x_s)^2+\frac{1}{2}\frac{d}{ds}\Big|_{s=0}\de_{\sX}(\gamma_{}(1),x_s)^2\geq 0
\end{align*}
Together with the previous observation it follows
\begin{align*}
&U_N(x_s)\left[\frac{d}{dt}\sigma_{\kappa^-_{\gamma_1^s}/\sN}^{\sscr{(t)}}(|\dot{\gamma}_1^s|)|_{t=1}+\frac{d}{dt}\sigma_{\kappa^-_{\gamma_0^s}/\sN}^{\sscr{(t)}}(|\dot{\gamma}_0^s|)|_{t=1}\right]\\
&\hspace{2cm}\geq \frac{d}{dt}\sigma_{\kappa^+_{\gamma_0^s}/\sN}^{\sscr{(t)}}(|\dot{\gamma}_0^s|)|_{t=0}{U_N(\gamma_{}(0))}+\frac{d}{dt}\sigma_{\kappa^+_{\gamma_1^s}/\sN}^{\sscr{(t)}}(|\dot{\gamma}_1^s|)|_{t=0}U_N(\gamma_{}(1))
\end{align*}
Now, let $s\rightarrow 0$. Since $\mms$ is locally compact and the length of $\gamma^s_i$ ($i=0,1$) is uniformily bounded, there exist uniformily converging subsequences of $\gamma^s_0$ and $\gamma_1^s$ w.r.t. $\de_{\infty}$.
The limits are denoted by $\varsigma_0$ and $\varsigma_1$.
By lower semi-continuity of
the length function one can see that the composition of $\varsigma_0$ and $\varsigma_1^-$ is again a geodesic between $\gamma(0)$ and $\gamma(1)$. 
Its constant speed reparametrization is denoted with $\varsigma:[0,1]\rightarrow X$ and its $1$-speed reparametrization is denoted with $\bar{\varsigma}$.
By construction we have $\varsigma(\frac{1}{2})=\gamma(\frac{1}{2})$.
Note, that $|\dot{\varsigma}|=\frac{1}{2}|\dot{\varsigma}_0|=\frac{1}{2}|\dot{\varsigma}_1|=|\dot{{\gamma}}|=2\epsilon$
and $\mbox{Im}\varsigma\subset B_{2h}(\bar{c}(r_i))$.
\smallskip\\
Recall that by Lemma \ref{uniform} $\frac{d^+}{dt}|_0\sigma_{\kappa}^{\sscr{(t)}}$ and $\frac{d^-}{dt}|_1\sigma_{\kappa}^{\sscr{(t)}}$ are lower and upper semi-continuous respectively with respect to $\kappa$ (Lemma \ref{uniform}). Hence
\begin{align}\label{autsch}
&U_N(\gamma({\textstyle\frac{1}{2}}))\left[\frac{d}{dt}\sigma_{\kappa^-_{\varsigma_1}/\sN}^{\sscr{(t)}}(|\dot{\varsigma}|_0)|_{t=1}+\frac{d}{dt}\sigma_{\kappa^-_{\varsigma_0}/\sN}^{\sscr{(t)}}(|\dot{\varsigma}_1|)|_{t=1}\right]\nonumber\\
&\hspace{1,5cm}\geq \frac{d}{dt}\sigma_{\kappa^+_{\varsigma_0}/\sN}^{\sscr{(t)}}(|\dot{\varsigma}_0|)|_{t=0}{U_N(\gamma_{}(0))}+\frac{d}{dt}\sigma_{\kappa^+_{\varsigma_1}/\sN}^{\sscr{(t)}}(|\dot{\varsigma}_1|)|_{t=0}U_N(\gamma_{}(1)).
\end{align}
%\\
%\\
%Let $\delta>0$ be arbitrary.
%Since $\kappa\circ\bar{\sigma}:[0,2\epsilon]\rightarrow \mathbb{R}$ is continuous, we can find ${h}>0$ such that
%\begin{align*}
%\max\kappa\circ\bar{\sigma}|_{(\epsilon-2h,\epsilon+2h)}-\min\kappa\circ\bar{\sigma}|_{(\epsilon-2h,\epsilon+2h)}<\delta.
%\end{align*}
%\\
%\\
Now, we use Taylor expansion of the coefficients. Recall from \cite{ketterer5} that
for $t$ fixed $f:h\mapsto \sigma_{\kappa}^{\sscr{(t)}}(h)$ is twice differentiable at $h=0$ and we have
\begin{align}\label{taylor}
h\in[0,{L}]\mapsto\sigma_{\kappa}^{(t)}(h)=t\left[1+\frac{1}{6}(1-t^2)\kappa(0)h^2\right]+o(h^2)_{\kappa}^{t}.
\end{align}
If $\overline{\kappa}\geq \kappa\geq \underline{\kappa}$, then
\begin{align*}
t\frac{1}{3}(1-t^2)(\underline{\kappa}-\overline{\kappa})h^2+o(h^2)^t_{\underline{\kappa}}\leq o(h^2)_{{\kappa}}^{t}&\leq t\frac{1}{3}(1-t^2)(\overline{\kappa}-\underline{\kappa})h^2+o(h^2)^t_{\overline{\kappa}}.
\end{align*}
Therefore, if $\gamma\in\mathcal{G}(X)$
\begin{align*}
\frac{d}{dt}\sigma_{\kappa_{{\gamma}}/\sN}^{\sscr{(t)}}(|\dot{\gamma}|)|_{t=0}\geq 1+\frac{1}{6}\kappa(\gamma(0))|\dot{\gamma}|^2+\frac{1}{3}(\underline{\kappa}_{\gamma}-\overline{\kappa}_{\gamma})|\dot{\gamma}|^2 + \frac{d}{dt}o(|\dot{\gamma}|^2)_{\underline{\kappa}_{\gamma}}^{t}|_{t=0}
\end{align*}
and 
\begin{align*}
\frac{d}{dt}\sigma_{\kappa_{{\gamma}}/\sN}^{\sscr{(1-t)}}(|\dot{\gamma}|)|_{t=0}\geq-1+\frac{2}{6}\kappa(\gamma(1))|\dot{\gamma}|^2+\frac{1}{3}(\underline{\kappa}_{\gamma}-\overline{\kappa}_{\gamma})|\dot{\gamma}|^2 + \frac{d}{dt}o(|\dot{\gamma}|^2)_{\underline{\kappa}_{\gamma}}^{1-t}|_{t=0}.
\end{align*}
\\
Now, consider $\varsigma_0$ and $\varsigma_1$. Let $\underline{\kappa}:=\min\kappa|_{B_{2h}(c(r_i))}$ and $\overline{\kappa}:=\max\kappa|_{B_{2h}(c(r_i))}$.
We plug this into (\ref{autsch}):
\begin{align*}
&U_N(\gamma({\textstyle\frac{1}{2}}))\left[2-\frac{1}{3}\kappa(\gamma(0))\epsilon^2-\frac{d}{dt}o(\epsilon^2)_{\underline{\kappa}_{}}^{1-t}|_{t=0}-\frac{1}{3}\kappa(\gamma(1))\epsilon^2-\frac{d}{dt}o(\epsilon^2)_{\underline{\kappa}_{}}^{1-t}|_{t=0}-\frac{2}{3}(\underline{\kappa}_{}-\overline{\kappa}_{})\epsilon^2\right]\\
&\hspace{1cm}\geq \left[1-\frac{1}{6}\kappa(\gamma_{}(\frac{1}{2}))\epsilon^2+\frac{d}{dt}o(\epsilon^2)_{\underline{\kappa}_{}}^{t}|_{t=0}+\frac{1}{3}(\underline{\kappa}_{}-\overline{\kappa}_{})\epsilon^2\right]{U_N(\gamma_{}(0))}\\
&\hspace{2cm}+
\left[1-\frac{1}{6}\kappa(\gamma_{}(\frac{1}{2}))\epsilon^2+\frac{d}{dt}o(\epsilon^2)_{\underline{\kappa}_{}}^{t}|_{t=0}+\frac{1}{3}(\underline{\kappa}_{}-\overline{\kappa}_{})\epsilon^2\right]U_N(\gamma_{}(1)).
\end{align*}
Rearranging the terms yields
\begin{align*}
&{2U_N(\gamma({\textstyle\frac{1}{2}}))-U_N(\gamma_{\epsilon}(0))-U_N(\gamma_{\epsilon}(1))}\\
&\hspace{0.5cm}\geq -U_N(\gamma(\frac{1}{2}))\epsilon^2\left[\frac{1}{3}\kappa(\gamma(0))+\frac{1}{3}\kappa(\gamma(1))\right]
-\frac{1}{6}\kappa(\gamma_{}(\frac{1}{2}))\epsilon^2U_N(\gamma_{\epsilon}(0))\\
&\hspace{0.5cm} -\frac{1}{6}\kappa(\gamma_{}(\frac{1}{2}))\epsilon^2U_N(\gamma_{\epsilon}(1))+ o(\epsilon^2)+\frac{1}{3}(\underline{\kappa}_{}-\overline{\kappa}_{})\epsilon^2\left[U_N(\gamma_{}(1))+U_N(\gamma_{}(0))+2U_N(\gamma(\frac{1}{2}))\right]
\end{align*}
In terms of the geodesic $c$ this is
\begin{align*}
&2U_N(\bar{c}(\hat{r}))-U_N(\bar{c}(\hat{r}-\epsilon))-U_N(\bar{c}(\hat{r}+\epsilon))\\
&\hspace{0.5cm}\geq -U_N(\bar{c}(\hat{r}))\left[\frac{1}{3}\kappa(\bar{c}(\hat{r}-\epsilon))\epsilon^2+\frac{1}{3}\kappa(\bar{c}(\hat{r}+\epsilon))\epsilon^2\right]\\
&\hspace{0.5cm}-\frac{1}{6}\kappa(\bar{c}(\hat{r}))\epsilon^2U_N(\bar{c}(\hat{r}-\epsilon))-\frac{1}{6}\kappa(\bar{c}(\hat{r}))\epsilon^2U_N(\bar{c}(\hat{r}+\epsilon)) + o(\epsilon)\\
&+\frac{1}{3}(\underline{\kappa}_{}-\overline{\kappa}_{})\epsilon^2\left[U_N(\bar{c}(\hat{r}-\epsilon))+U_N(\bar{c}(\hat{r}+\epsilon))+2U_N(\bar{c}(\hat{r}))\right].
\end{align*}
Deviding by $\epsilon^2>0$, multiplication with $\phi\in C^{\infty}_0((0,\theta))$ such that $\phi\geq 0$, integration with respect to $\hat{r}$, a change of variables and taking the limit $\epsilon\rightarrow 0$ yields
\begin{align}\label{oo}
\int_0^{\theta} U_N(\bar{c}(t))\phi''(t)dt\leq \int_0^{\theta}[\kappa(t)+\frac{4}{3}\delta]U_N(\bar{c}(t))\phi(t)dt.
\end{align}
Recall that $U_{\sN}$ is upper semi-continuous.
Since $\delta$ was arbitrary, the theorem follows from the characterization result of Proposition \ref{central}.\smallskip
\\
Finally, if $\kappa:X\rightarrow \mathbb{R}$ is lower semi-continuous, we choose $\kappa_n\uparrow \kappa$ pointwise for $\kappa_n$ bounded and continuous. By monotonicity the assumptions are satisfied for $\kappa_n$ instead of $\kappa$ for each $n\in \mathbb{R}$.
Hence, we can apply the first part of the proof, and we obtain (\ref{oo}) foe $\kappa_n$. But by the theorem of monotone convergence, this differential inequality still holds for $\kappa$.
\end{proof}
\begin{corollary}
If $(M,g_{\sM})$ is a Riemannian manifold, $f\in C^{2}(M)$ and $\kappa\in C(M)$, the following statements are equivalent
\begin{itemize}
 \item[(i)] For every $x_0\in \overline{D(f)}$ there exists an $EVI_{\kappa,\sN}$ gradient flow curve $(x_s)_{s\in(0,\infty)}$ starting in $x_0$.
 \smallskip
 \item[(ii)] $f$ is $(\kappa,N)$-convex.
\end{itemize}

\end{corollary}

\section{Reduced and entropic curvature dimension condition}
In this section we introduce an entropic curvature-dimension condition for metric measure spaces and semi-continuous lower curvature bound $\kappa$. 
For this purpose we will apply the results of the previous section in the following context.
\smallskip
\smallskip\\
For $\mu\in\mathcal{P}_2(X)$ we define the \textit{relative entropy} by 
\begin{align*}
\Ent(\mu):=\int\rho\log\rho d\m_{\sX}
\end{align*}
if $\mu\in \mathcal{P}(\m_{\sX})$ and $(\rho\log\rho)_+$ is integrable. Otherwise, we set $\Ent(\mu)=\infty$. Moreover, for $N\in (0,\infty)$ we introduce the functional $U_N:\mathcal{P}_2(X)\rightarrow [0,\infty]$ by
\begin{align*}
U_{\sN}(\mu):=\exp\left(-{\textstyle \frac{1}{N}}\Ent(\mu)\right).
\end{align*}
If we assume the following volume growth condition 
\begin{align}\label{growthcondition}
\int e^{-c\de(p,x)^2}d\m_{\sX}(x)<\infty
\end{align}
it is well known that $\Ent$ does not take the value $-\infty$ on $\mathcal{P}_2(X)$ and $\Ent$ is lower semi-continuous with respect to $W_2$. 
\begin{definition}
A metric measure space $(X,\de_{\sX},\m_{\sX})$ satifies the \textit{entropic curvature-dimension condition} $CD^e(\kappa,N)$ for some lower semi-continuous function $\kappa$ and $N\geq 0$ if 
for any pair $\mu_0,\mu_1\in D(\Ent)$ with compact support there exists a $L^2$-Wasserstein geodesic $\Pi$ connection $\mu_0$ and $\mu_1$ such that for all $t\in [0,1]$
\begin{align}\label{entropic}
U_N(\mu_t)\geq \sigma_{\kappa^-_{\Pi}/\sN}^{\sscr{(1-t)}}(\Theta)U_N(\mu_0)+\sigma_{\kappa^+_{\sPi}/\sN}^{\sscr{(t)}}(\Theta)U_N(\mu_1)
\end{align}
where $(e_t)_{\star}\Pi=\mu_t$, $\Theta=W_2(\mu_0,\mu_1)$ and 
\begin{align*}
\kappa_{\sPi}(t\Theta)=\frac{1}{\Theta^2}\int \kappa(e_t(\gamma)|\dot{\gamma}|^2d\Pi(\gamma).
\end{align*}
Recall that $\kappa_{\Pi}$ is lower semi-continuous by Remark \ref{ody}. 
If (\ref{entropic}) holds for any geodesic $\Pi\in\mathcal{P}(\mathcal{G}(X))$ we say that $(X,\de_{\sX},\m_{\sX})$ is a strong $CD^e(\kappa,N)$ space.
$\mu\in D(\Ent)$ implies that $\mu\in \mathcal{P}_2(\m_{\sX})$.
\begin{remark}
The coefficient $\sigma_{\kappa_{\Pi}/\sN}^{\sscr{(t)}}(\Theta)$ solves 
$
v''(t)+\int \kappa(e_t(\gamma))/N|\dot{\gamma}|^2d\Pi(\gamma)v(t)=0.
$
Therefore, we can write $\sigma_{\kappa_{\Pi}\Theta^2/\sN}^{\sscr{(t)}}$.
\end{remark}
\begin{remark}
The entropic curvature-dimension condition is not $(K,N)$-convexity of $U_{\sN}$ for some function $K$ on $\mathcal{P}_2(X)$. More precisely, it is 
$(\kappa_{\Pi},N)$-convexity on Wasserstein geodesics $\Pi$ where $\kappa_{\Pi}$ depends on the geodesic $\Pi$.
\end{remark}
%\begin{definition}
%Consider $\mms$ that satisfies $CD^e(\kappa,N)$. We define the \textit{effective diameter} of $\mms$ as
%\begin{align*}
%\pi_{\frac{\kappa}{N}}=\sup\left\{W_2(\mu_0,\mu_1):\exists \Pi\in\mathcal{G}(\mathcal{P}^2(X)) \mbox{ between $\mu_0$ and $\mu_1$ s.t. }\sigma^{\sscr{(t)}}_{\kappa_{\Pi}/N}(\Theta)<\infty\right\}.
%\end{align*}
%As in Proposition 5.7 in \cite{ketterer5}, one sees that $\pi_{\frac{\kappa}{N}}=\diam_{\sX}$.
%\end{definition}
\begin{lemma}\label{ufufuf} Let $\mms$ be a metric measure space, and let $\kappa$ be lower semi-continuous. Then
\begin{align*}
\liminf_{i\rightarrow\infty}\sigma_{\kappa_{\Pi_i}/\sN}^{\sscr{(t)}}(\Theta)\geq \sigma_{\kappa_{\Pi}/\sN}^{\sscr{(t)}}(\Theta)\mbox{ for each }t\in [0,1]
\end{align*}
if $\Pi_i$ converges weakly to $\Pi$.
%$\Pi\in\mathcal{P}(\mathcal{G}(X))\mapsto\sigma_{\kappa_{\Pi}/\sN}^{\sscr{(t)}}(\Theta)$ is lower semi-continuous with respect to weak convergence. 
%If $\kappa$ is continuous and $\mms$ is compact, then $\Pi\mapsto\sigma_{\kappa_{\Pi}/\sN}^{\sscr{(t)}}(\Theta)$ is continuous.
\end{lemma}
\begin{proof} One can easily check that for each $t\in[0,1]$ $\Pi\mapsto \int\kappa(\gamma(t))|\dot{\gamma}|^2d\Pi(\gamma)=\kappa_{\Pi}(t\Theta)\Theta^2$ is lower semi-continuous w.r.t. weak convergence. 
Then, the statement follows from Lemma \ref{uff}.
%In the case when $\kappa$ is continuous and $\mms$ is compact, then 
%$\Pi\mapsto\kappa_{\Pi}(t\Theta)\Theta^2$ is even continuous, by definition of weak convergence. Additionally, 
%$t\in [0,1]\mapsto \kappa_{\Pi}(t\Theta)\Theta^2$ is continuous for any $\Pi$. Therefore, $\kappa_{\Pi}\Theta$ converges uniformily if $\Pi$ converges w.r.t. weak convergence.
%Hence, $\Pi\mapsto\sigma_{\kappa_{\Pi}/\sN}^{\sscr{(t)}}(\Theta)$ changes continuously as well.
\end{proof}
\end{definition}
\begin{lemma}
Let $(X,\de_{\sX},\m_{\sX})$ be a metric measure space satisfying $CD^e(\kappa,N)$ for some lower semi-continuous $\kappa$ and $N>0$.
\begin{itemize}
 \item[(i)] If $\kappa'$ is admissible with $\kappa'\leq \kappa$ and $N'\geq N$, then $(X,\de_{\sX},\m_{\sX})$ satisfies $CD^e(\kappa',N)$. \\
 In particular, $(X,\de_{\sX},\m_{\sX})$ satisfies $CD(\kappa,\infty)$ in the sense of \cite{sturmvariable}.
 \smallskip
 \item[(ii)] Let $V:X\rightarrow \mathbb{R}$ be a measurable function that is bounded from below and that is $(\kappa',N')$-convex in sense of Definition \ref{ahr}
 for some admissible $\kappa'$ and $N'>0$.
\end{itemize}
\end{lemma}
\begin{proof}
(i)\ The first part is an immediate consequence of Proposition \ref{monotonicity}. The second part follows by monotonicity in $N>0$. 
Consider
\begin{align}\label{blabla}
&(1-t)\Ent(\mu_0)+t\Ent(\mu_1)-\Ent(\mu_t)\nonumber\\
&=\lim_{N\rightarrow\infty}\left(-(1-t)N\left(U_N(\mu_0)-1\right)-tN\left(U_N(\mu_1)-1\right)+N\left(U_N(\mu_t)-1\right)\right)\nonumber\\
&\leq \lim_{N\rightarrow\infty}\left[\left(\sigma_{\kappa^-_{\Pi}/\sN}^{\sscr{(1-t)}}(\Theta)-(1-t)\right)N+\left(\sigma_{\kappa^+_{\sPi}/\sN}^{\sscr{(t)}}(\Theta)-t\right)N\right]\nonumber\\
&\leq \lim_{N\rightarrow\infty}\underbrace{N\left[\sigma_{\kappa^-_{\Pi}/\sN}^{\sscr{(1-t)}}(\Theta)+\sigma_{\kappa^+_{\sPi}/\sN}^{\sscr{(t)}}(\Theta)-1\right]}_{=:v(t)}
\end{align}
For large $N$ the function $v$ solves 
\begin{align*}
v''+\int \kappa(e_t(\gamma))|\dot{\gamma}|^2d\Pi(\gamma)\leq 0.
\end{align*}
Therefore the RHS in (\ref{blabla}) is smaller or equal than $\int \kappa(e_t(\gamma))|\dot{\gamma}|^2d\Pi(\gamma)$.
\medskip\\
(ii) We define $\overline{V}:\mathcal{P}_2(X)\rightarrow (-\infty,\infty]$ by $\overline{V}(\mu)=\int V d\mu$. We show that $\overline{V}$ is $(\kappa',N')$-convex.
Recall that by Lemma \ref{grrr} $g:(x,y,\kappa)\mapsto \log \left(\sigma_{\kappa_{\gamma}^-}^{(1-t)}e^x+\sigma_{\kappa_{\gamma}^+}^{(t)}e^y\right)$ is convex. 
Hence, if $\Pi$ is a geodesic in $\mathcal{P}_2(X)$, then 
\begin{align*}
-\frac{1}{N'}\overline{V}(\mu_t)&=-\frac{1}{N'}\int V(e_t(\gamma))d\Pi(\gamma)\\
&\geq \int g\left(-{\textstyle\frac{1}{N'}}V(e_0(\gamma)),-{\textstyle\frac{1}{N'}}V(e_1(\gamma)),\kappa'_{\gamma}|\dot{\gamma}|^2/N'\right)d\Pi(\gamma)\\
&\geq g\left(-{\textstyle\frac{1}{N'}}\overline{V}(\mu_0),-{\textstyle\frac{1}{N'}}\overline{V}(\mu_1),\kappa'_{\sPi}\Theta^2\right).
\end{align*}
After taking the exponential this $(\kappa',N')$-convexity of $\overline{V}$.
Finally, since by lower boundedness of $V$ we have $\mathcal{P}(e^{-V}\m_{\sX})\subset \mathcal{P}(m_{\sX})$, 
and since $\Ent_{e^{-V}\m_{\sX}}(\mu)=\Ent_{\m_{\sX}}(\mu)+\overline{V}(\mu)$, we obtain the result by Lemma \ref{sums}.
\end{proof}

\begin{definition}[Minkowski content]
Consider $x_0\in X$ and $B_r(x_0)\subset X$. Set $v(r)=\m_{\sX}(\bar{B}_r(x_0))$. The Minkowski content of $\partial B_r(x_0)$ (the $r$-sphere around $x_0$) is defined as
\begin{align*}
s(r):=\limsup_{\delta\rightarrow 0}\frac{1}{\delta}\m_{\sX}(\bar{B}_{r+\delta}(x_0)\backslash B_r(x_0)).
\end{align*}
\end{definition}

\begin{theorem}\label{useful}
Assume $(X,\de_{\sX},\m_{\sX})$ satisfies $CD^e(\VK,N)$ for a lower semi-continuous function $\VK$ and $N\in[1,\infty)$.
Then, $(X,\de_{\sX})$ is a proper metric space, each bounded set has finite measure and satisfies a doubling property, and either $\m_{\sX}$ is supported by one point or all points and all sphere have mass $0$. 
\medskip\\
In particular, if $N>1$ then for each $x_0\in X$, for all $0<r<R$ and $\underline{\VK}\in\mathbb{R}$ such that $\VK|_{B_R(x_0)}\geq \underline{\VK}$ and $R\leq \pi\sqrt{N/\underline{\VK}\vee 0}$, we have
\begin{align}\label{one}
\frac{s(r)}{s(R)}\geq \frac{\sin_{\underline{\sk}/\sN}^{\sN}r}{\sin_{\underline{\sk}/\sN}^{\sN}R}\ \ \ \ \& \ \ \ \
\frac{\m_{\sX}(B_r(x_0)}{\m_{\sX}(B_R(x_0)}\geq \frac{\int_0^r\sin_{\underline{\sk}/\sN}^{\sN}tdt}{\int_0^R\sin_{\underline{\sk}/\sN}^{\sN}tdt}.
\end{align}
If $N=1$ and $\VK\leq 0$, then
$
\frac{s(r)}{s(R)}\geq 1$ and $\frac{\m_{\sX}(B_r(x_0)}{\m_{\sX}(B_R(x_0)}\geq \frac{r}{R}.
$
\end{theorem}
\begin{proof}
Theorem follows as in the proof of Theorem 5.3 in \cite{ketterer5}.
\end{proof}
\begin{remark}\label{remark}
The estimates in the previous theorem are not sharp. Though, it is enough to prove that $(X,\de_{\sX},\m_{\sX})$ satisfies a doubling property and
has finite Hausdorff dimension bounded from above by $N$.
See Corollary 5.4, Corollary 5.5 in \cite{ketterer5} for the proofs. The entropic curvature-dimension condition $CD^e(\kappa,N)$ yields that $(\supp\m_{\sX},\de_{\sX})$ is
a length space. Therefore, by local compactness and completness it is geodesic, and therefore $\mathcal{P}_2(\supp\m_{\sX})$ is geodesic as well. 
Additionally, if $\kappa$ is bounded from below, the volume growth estimate (\ref{one}) implies that 
\begin{align*}
\int e^{-c\de_{\sX}(p,x)^2}d\m_{\sX}(x)<\infty
\end{align*}
for some point $p\in X$ and $c>0$, and in particular, $\Ent>-\infty$.
\end{remark}
\noindent
Recall that a sequence of $\mmsi_{i\in\mathbb{N}}$ with $\m_{\sX_i}(X_i)<\infty$ converges in Gromov sense to a metric measure space $\mms$ if there exists a metric space $(Z,\de_{Z})$ and isometric embeddings 
$\iota_i,\iota:X_i\rightarrow Z$ for $i\in\mathbb{N}$ 
such that $(\iota_i)_{\star}\m_{\sX_i}$ converges weakly to $(\iota)_{\star}\m_{\sX}$. This can be equivalently defined in terms of Sturm's transportation distance $\mathbb{D}$ (see \cite{stugeo1}).
\begin{definition}
Let $\mmsi_{i\in\mathbb{N}}$ be metric measure spaces converging in Gromov sense to a metric measure space $\mms$. 
Let $\kappa_i,\kappa:X_i,X\rightarrow \mathbb{R}$ be lower semi-continuous functions. 
We say
\begin{align*}
\liminf_{i\rightarrow\infty}\VK_i\geq \VK
\end{align*}
if for each $\eta>0$ there exists $i_{\eta}\in\mathbb{N}$ such that $\VK_i(x)\geq \VK(f_i(x))-\eta$ if $i\geq i_{\eta}$ for each $x\in X_i$.
We say $(\kappa_i)$ converges uniformily to $\kappa$
if for each $\eta>0$ there exists $\delta$ and $i_{\eta}\in\mathbb{N}$ such that $|\kappa_i(x)-\kappa(y)|<\eta$ if $i\geq i_{\eta}$ and $\de_{\sZ}(x,y)<\delta$.
\end{definition}
\begin{theorem}\label{mmmm}
Let $(X_i,\de_{\sX_i},\m_{\sX_i})_{i\in\mathbb{N}}$ be a sequence of metric measure spaces satisfying the condition $CD^e(\VK_i,N_i)$ respectively for lower semi-continuous functions $\VK_i$ and $N_i\in[1,\infty)$. 
Assume $\mmsi$ converges to a metric measure space $(X,\de_{\sX},\m_{\sX})$ 
in 
Gromov sense, and consider an admissible function $\VK:X\rightarrow \mathbb{R}$ and $N\in[1,\infty)$ such that 
\begin{align}\label{zzz}
\liminf_{i\rightarrow\infty}\VK_i\geq \VK\geq K\ \ \& \ \ \limsup_{i\rightarrow \infty}N_i\leq N
\end{align}
Then $(X,\de_{\sX},\m_{\sX})$ satisfies $CD^e(\VK,N)$.
\end{theorem}
\begin{proof}
Since $\m_{\sX_i}$ and $\m_{\sX}$ are finite, the Entropy functional on $\mathcal{P}_2(Z)$ is lower semi-continuous.
Then, the proof is the same as the proof of corresponding stability results in \cite{gmsstability}. The only additional information one needs is that 
$$\liminf_{i\rightarrow \infty}\int\kappa_i(\gamma(t))|\dot{\gamma}|^2d\Pi_i(\gamma)\geq \int\kappa(\gamma(t))|\dot{\gamma}|^2d\Pi(\gamma)$$
if $(\Pi_i)_{i\in\mathbb{N}}$ with $\Pi_i\in \mathcal{P}(\mathcal{G}(X_i))$ converges weakly in $\mathcal{P}(\mathcal{G}(Z)$ to $\Pi\in \mathcal{P}(\mathcal{G}(X))$.
Then, we can apply Lemma \ref{uff} to obtain the result.
\end{proof}

\begin{definition}[\cite{ketterer5}]\label{bigg}
Consider a lower semi-continuous function $\VK:X\rightarrow \mathbb{R}$, and let $N\in\mathbb{R}$ with $N\geq 1$.
$(X,\de_{\sX},\m_{\sX})$ satisfies the \textit{reduced curvature-dimension condition}
$CD^*(\VK,N)$ if and only if for each pair $\nu_0,\nu_1\in \mathcal{P}_2(X,\m_{\sX})$ with bounded support
there exists a dynamical optimal coupling $\Pi$ of $\nu_0=\varrho_0d\m_{\sX}$ and $\nu_1=\varrho_1d\m_{\sX}$ and a geodesic $(\nu_t)_{t\in[0,1]}\subset\mathcal{P}_2(X,\m_{\sX})$, such that
\begin{align}\label{curvaturedimension}
%\int_{X}\varrho_t(x)^{-\frac{1}{N'}}d\nu_t(x)
S_{N'}(\nu_t)\leq{\textstyle -}\!\!\int\Big[\sigma_{\VK^{\sscr{-}}_{\gamma}/N'}^{\sscr{(1-t)}}(|\dot{\gamma}|)\varrho_0\left(e_0(\gamma)\right)^{-\frac{1}{N'}}+\sigma_{\VK^{\sscr{+}}_{\gamma}/N'}^{\sscr{(t)}}(|\dot{\gamma}|)\varrho_1\left(e_1(\gamma)\right)^{-\frac{1}{N'}}\Big]d\Pi(\gamma)
\end{align}
for all $t\in [0,1]$ and all $N'\geq N$.
\end{definition}
\begin{definition}
We say that a metric measure space $\mms$ is \textit{essentially non-branching} if for any optimal dynamcial plan $\Pi\in\mathcal{P}(\mathcal{G}(X))$ between absolutely continuous probability measures
is supported on set of non-branching geodesics. More precisely, there exists $A\subset \mathcal{G}(X)$ such that $\Pi(A)=1$ and for all $\gamma, \gamma'\in A$ we have the following property
\begin{align*}
\gamma(t)=\gamma'(t)\ \ \mbox{ for all }t\in [0,\epsilon] \ \mbox{ for some }\epsilon>0\ \Longrightarrow \ \gamma=\gamma'.
\end{align*}
\end{definition}
\begin{lemma}[\cite{erbarkuwadasturm}]\label{ufuf}
Let $\mms$ be an essentially non-branching metric measure space, and let $\Pi$ be an optimal dynamical coupling. Assume $\Pi=\sum_{i=1}^n\alpha_i\Pi_i$ for optimal dynamical couplings $\Pi_i$. If $(e_0)_{\star}\Pi_i$ are
mutually singular, then the family $(e_t)_{\star}\Pi_i$ is mutually singular as well.
\end{lemma}
\begin{theorem}\label{nonbranching}
Let $(X,\de_{\sX},\m_{\sX})$ be an essentially non-branching metric measure space, and let $\kappa$ be a lower semi-continuous function and $N\geq 1$. Then the following statements are equivalent.
\begin{itemize}
 \item[(i)] $\mms$ satisfies $CD^*(\kappa,N)$.
 \smallskip
 \item[(ii)] For each pair $\mu_0,\mu_1\in\mathcal{P}_2(\m_{\sX})$ with bounded support there exists an optimal dynamical transference plan $\Pi$ with $(e_t)_{\star}\Pi\in\mathcal{P}_2(\m_{\sX})$ such that
\begin{align}\label{something}
\varrho_t(\gamma_t)^{-\frac{1}{N}}\geq \sigma_{\sk^-_{\gamma}/\sN'}^{\sscr{(1-t)}}(|\dot{\gamma}|)\varrho_0(\gamma_0)^{-\frac{1}{N}}+\sigma_{\sk^+_{\gamma}/\sN}^{\sscr{(t)}}(|\dot{\gamma}|)\varrho_1(\gamma_1)^{-\frac{1}{N}}.
\end{align}
for all $t\in[0,1]$ and $\Pi$-a.e. $\gamma\in\mathcal{G}(X)$. $\varrho_t$ is the density of the push-forward of $\Pi$ under the map $\gamma\mapsto \gamma_t$.
\smallskip
 \item[(iii)] $\mms$ satisfies $CD^e(\kappa,N)$.
\end{itemize}
\end{theorem}
\begin{proof}
``(i) $\Leftrightarrow$ (ii)'': The same equivalence was shown in \cite{ketterer5} for the curvature-dimension condition $CD(\kappa,N)$ provided $\mms$ is a non-branching metric measure space, and it is obvious that
same proof also works for the condition $CD^*(\kappa,N)$. If we assume that $\mms$ is essentially non-branching, we can apply the same proof using Lemma \ref{ufuf}. (Compare also with Theorem 3.12 in \cite{erbarkuwadasturm})
\smallskip
\\
``(ii) $\Rightarrow$ (iii)'': By the same argument as in \cite[Lemma 2.11]{bast} one can show that the statement (ii) holds for all $\mu_0,\mu_1\in \mathcal{P}_2(\m_{\sX})$.
Now, let $\Pi$ be an optimal dynamical plan between $\mu_0$ and $\mu_1$ satisfying (\ref{something}). It follows
\begin{align}\label{something1}
-\frac{1}{N}\log\varrho_t(\gamma_t)\geq \log\left[\sigma_{\sk^-_{\gamma}/\sN'}^{\sscr{(1-t)}}(|\dot{\gamma}|)\varrho_0(\gamma_0)^{-\frac{1}{N}}+\sigma_{\sk^+_{\gamma}/\sN}^{\sscr{(t)}}(|\dot{\gamma}|)\varrho_1(\gamma_1)^{-\frac{1}{N}}\right].
\end{align}
Again, Lemma \ref{grrr} says that $g_t:(x,y,\kappa)\mapsto \log \left(\sigma_{\kappa_{\gamma}^-}^{(1-t)}e^x+\sigma_{\kappa_{\gamma}^+}^{(t)}e^y\right)$ is convex.
Therefore, integrating (\ref{something1}) with respect to $\Pi$ and applying Jensen's inequality yields
\begin{align*}
-\frac{1}{N}\Ent(\mu_t)\geq g_t\left(-\frac{1}{N}\Ent(\mu_0),-\frac{1}{N}\Ent(\mu_1),\int\kappa(e_t(\gamma)/N |\dot{\gamma}|^2d\Pi\right).
\end{align*}
Hence, statement (iii) follows by taking the exponential on both sides.
\medskip\\
``(iii) $\Rightarrow$ (ii)'': We roughly follow the argument in the proof of \cite[Theorem 3.12]{erbarkuwadasturm}. Let $\mu_0,\mu_1\in\mathcal{P}_2(\m_{\sX})$ be with bounded support, and let $\Pi$ be an optimal dynamical coupling 
between $\mu_0$ and $\mu_1$. Let $\left\{M_n\right\}_{n\in\mathbb{N}}$ be an $\cap$-stable generator of the Borel $\sigma$-field of $(X,\de_{\sX})$ .
%with $\m_{\sX}(\partial M_n)=0$ for all $n\in\mathbb{N}$.
For each $n$ we define a disjoint covering of $X$ of $2^n$ sets by $L_I=\bigcap_{i\in I} M_i \cap \bigcap_{i\in I^c}^{} M_i^c $ where $I\subset \left\{1,\dots,n\right\}$.
\smallskip\\
We fix $n$, and we define $A_{i,j}=\left\{\gamma\in\mathcal{G}(X): (\gamma(0),\gamma(1))\in L_i\times L_j\right\}$ for $i,j=1,\dots,2^n$, 
and probability measures $\mu_t^{i,j}={\Pi(A_{i,j})}^{-1}(e_t)_{\star}\Pi|_{A_{i,j}}$ for $t=0,1$ provided $\Pi(A_{i,j})>0$. By (iii) there are optimal dynamical plan $\Pi^{i,j}$ between $\mu_0^{i,j}$ and $\mu_1^{i,j}$
such that (\ref{entropic}) holds, and $\mu_t^{i,j}=(e_t)_{\star}\Pi^{i,j}$ are absolutely continuous w.r.t. $\m_{\sX}$. Then, we define an optimal dynamical coupling between $\mu_0$ and $\mu_1$ by 
\begin{align*}
\Pi^n=\sum_{i,j=1}^{2^n}\Pi(A_{i,j})\Pi^{i,j}.
\end{align*}Since $\mu_{t}^{i,j}$ are absolutely continuous, $\mu_t^n=(e_t)_{\star}\Pi^n$ is absolutely continuous as well, and 
since the measures $\mu_0^{i,j}$ are mutually singular, $\mu_t^{i,j}=\rho_t^{i,j}d\m_{\sX}$ are mutually singular as well by Lemma \ref{ufuf}. Therefore, $\rho^n_t=\sum\Pi(A_{i,j})\rho_t^{i,j}$ with $\rho^n_t|_{A_{i,j}}=\Pi(A_{i,j})\rho_t^{i,j}$.
Hence, (\ref{entropic}) becomes after taking logarithms (set $\frac{\Pi(A_{i,j})^{-1}}{N}=\alpha_{i,j}$)
\begin{align}\label{mo}
&-\alpha_{i,j}\int_{A_{i,j}}\log \rho_t^n((e_t)\gamma)d\Pi^n(\gamma)\nonumber\\
& \geq g_t\left[-\alpha_{i,j}^{-1}\int_{A_{i,j}}\log\rho_0(e_0(\gamma))d\Pi^n(\gamma),\dots ,-\alpha_{i,j}^{-1}\int_{A_{i,j}}\kappa((e_t)\gamma)|\dot{\gamma}|^2d\Pi^n(\gamma)\right]
\end{align}
Since $\mu_0,\mu_1$ have bounded support,
all measure under consideration are supported by a common compact subsets in $X$ and $\mathcal{G}(X)$ respectively independent of $n$. 
Therefore, up to extraction of subsequences Prohorov's yields that $\Pi^n$ converges weakly to a dynamical coupling $\tilde{\Pi}\in \mathcal{P}(\mathcal{G}(X))$ 
that is optimal by lower semi-continuity of the Wasserstein distance under weak convergence. 
%In particular, since $\m_{\sX}(\partial M_i)=0$ for all $i$, we have
%\begin{align*}
%\Pi(\left\{\gamma\in\mathcal{G}(X):\gamma(0)\in M_i,\gamma(1)\in M_j\right\})=
%\end{align*}
\smallskip\\
Now, note that the relative Entropy $\Ent_{\m_{\sX}|{e_t(A_{i,j})}}$ is lower semi-continuous. 
First, this implies that we can pass to the limit in the LHS of (\ref{mo}) for $n\rightarrow \infty$.  Second, by Lemma \ref{ufufuf}
\begin{align*}
\liminf_{n\rightarrow \infty}\sigma_{\tilde{\kappa}_n}^{(t)}\geq \sigma_{\tilde{\kappa}}^{(t)}
\end{align*}
where $\tilde{\kappa}(t):=\int_{A_{i,j}}\kappa((e_t)\gamma)|\dot{\gamma}|^2d\Pi(\gamma)$. Hence, (\ref{mo}) also holds if we replace $\Pi^n$ by $\Pi$. Finally, by convexity of $g_t$ and Jensen's inequality
(\ref{mo}) also holds when we
replace $A_{i,j}$ by $A$ where $A$ is a disjoint union of sets $A_{i,j}$ for $i,j\in\left\{1,\dots,2^n\right\}$ and $n\in \mathbb{N}$. 
Therefore, (\ref{mo}) holds for any set in the $\cap$-stable generatore, and consequently the inequality holds for $\Pi$-a.e. $\gamma$.
\end{proof}
\section{Riemannian curvature-dimension condition}
Let $\mms$ be a metric measure space. We will briefly repeat some concepts for calculus on metric measure spaces.
For any function $u:X\rightarrow \mathbb{R}$ in $L^2(\m_{\sX})$ the Cheeger energy $\ChX(u)$ can be defined by 
\begin{align*}
 \ChX(u)=\frac{1}{2}\inf\left\{\liminf_{h\rightarrow \infty}\int_{\sX}\left(\lip u_h\right)^2d\m_{\sX}: \left\|u_h-u\right\|_{L^2(\m_{\sX})}\rightarrow 0\right\}.
\end{align*}
The $L^2$-Sobolev space is given by $D(\ChX)=\left\{u\in L^2(\m_{\sX}):\ChX(u)<\infty\right\}$. 
An important fact is that $\Ch$ is not a quadratic form in general. 
\begin{definition}
We say that a metric measure space $\mms$ is \textit{infinetimally Hilbertian} if the associated Cheeger enegery is quadratic. 
\end{definition}
\begin{definition}
Let $\mms$ be a metric measure space, and let $\kappa:X\rightarrow \mathbb{R}$ be lower semi-continuous and bounded from below. We say that $\mms$ satisfies the 
Riemannian curvature-dimension condition $RCD^*(\kappa,N)$ for $N\geq 1$ if $\mms$ is infinitesimally Hilbertian and satisfies the condition $CD^e(\kappa,N)$.
\end{definition}

\noindent
In \cite{agslipschitz} the authors show that $\ChX$ can be represented by
\begin{align}\label{shan}
\Ch^{\sX}(u)=\frac{1}{2}\int_{\sX}|\nabla u|_w^2d\m_{\sX}&\hspace{6pt}\mbox{if  } u\in D(\Ch)
\end{align}
and $+\infty$ otherwise where $|\nabla u|_w:X\rightarrow [0,\infty]$ is Borel measurable and called the minimal weak upper gradient of $u$. 

In particular, $\ChX$ is convex and lower semi-continuous. This allows to define a Laplacian $L^{\sX}$ on $L^2(\m_{\sX})$ as the $L^2$-norm subdifferential of $\ChX$. 
$L^{\sX}$ is not a linear operator in general. 
Still, the classical theory of gradient flows of convex functionals in Hilbert spaces yields that for any $f\in L^2(\m_{\sX})$ there is a unique, locally absolutely continuous flow curve $(f_t)_{t>0}$ starting at $f$ such that 
\begin{align*}
 \frac{d^+}{dt}f_t=Lf_t\ \mbox{ for all }t>0.
\end{align*}
On the other hand, one can study the metric gradient flow of the relative entropy $\Ent$ in $\mathcal{P}_2(X)$ in the sense of the energy dissipation inequality (\ref{dissipation}). 
In \cite{agsheat} the authors prove that for any $\mu\in D(\Ent)$ there exists a unique gradient flow curve in this sense provided the metric measure space satisfies a curvature bound condition in the sense of Lott, Sturm and Villani. 
This also gives a semi-group $\mathcal{H}_t$ on $\mathcal{P}_2(X)$. Then, the main result in \cite{agsheat} is the following indentification between the two gradiend flows.
\begin{theorem}
Let $\mms$ be a $CD(K,\infty)$ space and let $f\in L^2(\m_{\sX})$ such that $d\mu=fd\m_{\sX}\in \mathcal{P}_2(X)$. Then 
$
d\mathcal{H}_t\mu=(P_t f)d\m_{\sX}.
$
\end{theorem}

\begin{definition}\label{evikn}
We say that a metric measure space $\mms$ satisfies the \textit{evolution-variational inequality} - $EVI_{\kappa,\sN} $ - for some lower semi-continuous function $\kappa:X\rightarrow \mathbb{R}$
and $N\geq 1$ if for every $\overline{\mu}\in\mathcal{P}_2(X)$ there exists a curve $(\mu^s)_{s\in (0,\infty)}$ in $D(\Ent)$ with $\lim_{s\rightarrow 0}\mu^s=\overline{\mu}$, and for each
$s>0$ and each $\nu\in\mathcal{P}_2(X)$ there exists a geodesic $\Pi^s\in\mathcal{P}(\mathcal{G}(X))$ between $\mu^s$ and $\nu$ such that 
\begin{align*}
\frac{d}{dt}\Big|_{t=1}\sigma^{(t)}_{\kappa_{\sPi^s}/\sN}(\Theta^s)>-\infty\ \mbox{ and } \ \frac{d}{dt}\Big|_{t=0}\sigma^{(t)}_{\kappa_{\sPi^s}/\sN}(\Theta^s)<\infty
\end{align*} for each $s>0$
and
\begin{align*}-
\frac{1}{N}\frac{d}{ds}\frac{1}{2}W_2(\mu^s,\nu)^2 + \frac{d}{dt}\Big|_{t=1}\sigma^{(t)}_{\kappa_{\sPi^s}/\sN}(\Theta^s)\geq \frac{d}{dt}\Big|_{t=0}\sigma^{(t)}_{\kappa_{\sPi^s}/\sN}(\Theta^s)\frac{U_N(\nu)}{U_N(\mu^s)}.
\end{align*}
where $\Theta^s=W_2(\mu^s,\nu)$. We also say that $\mu^s$ is a \textit{$L^2$-Wasserstein $EVI_{\kappa,\sN}$ gradient flow curve}.
\end{definition}

\begin{remark}
In \cite{sturmvariable} Sturm makes the following definition. 
We say that a metric measure space $\mms$ satisfies $EVI_{\kappa,\infty} $ if for every $\overline{\mu}\in\mathcal{P}_2(X)$ there exists a curve $(\mu^s)_{s\in (0,\infty)}$ in $D(\Ent)$ and 
$\Pi^s\in\mathcal{P}(\mathcal{G}(X))$ between $\mu^s$ and $\nu$ as in Definition \ref{evikn} such that 
\begin{align}\label{evi}
\frac{1}{2}\frac{d}{ds}W_2(\mu^s,\nu)^2+\left[\int_0^1 (1-t)\kappa_{\sPi^s}(t\Theta^s)dt\right]W_2(\mu^s,\nu)^2\leq \Ent(\mu^s)-\Ent(\nu)
\end{align}
holds for a.e. $t>0$ where $\Theta^s:=W_2(\mu^s,\nu)$. $\mu^s$ is a $L^2$-Wasserstein $EVI_{\kappa,\infty}$ gradient flow curve.
\end{remark}
\begin{remark}
The implications of Lemma \ref{monot} hold as well on the level of Wasserstein gradient flows. In particular $EVI_{\kappa,N}$ implies $EVI_{\kappa,\infty}$.
\end{remark}
\begin{theorem}\label{B}
Let $\mms$ be a metric measure space with $\supp \m_{\sX}=X$, and let $\kappa$ be a lower semi-continuous function with $\kappa\geq K\in\mathbb{R}$ and $N\geq 1$. 
Then the following tree statements are equivalent:
\begin{itemize}
 \item[(i)] $\mms$ is infinitesimally Hilbertian and satisfies $CD^*(\kappa,N)$.
 \smallskip
 \item[(ii)] $\mms$ is infinitesimally Hilbertian and satisfies $CD^e(\kappa,N)$.
 \smallskip
 \item[(iii)] $\mms$ is a length space that satisfies the volume growth condition (\ref{growthcondition}) and $EVI_{\kappa,\sN}$.
\end{itemize}
\end{theorem}
\begin{proof}
``(i) $\Leftrightarrow$ (ii)'': Both conditions - $CD^*(\kappa,N)$ and $CD^e(\kappa,N)$ - imply a condition $CD(K,\infty)$ 
(see also \cite{ketterer5}). 
Therefore, from \cite{agsriemannian}
follows that $\mms$ satisfies $EVI_K$. Hence, $\mms$ is essentially non-branching by \cite{rajalasturm}, and Theorem \ref{nonbranching} yields the equivalence of $CD^*(\kappa,N)$ and $CD^e(\kappa,N)$.
\\
\\
``(ii) $\Rightarrow$ (iii)'': By Remark (\ref{remark}), $(X,\de_{\sX})$ is a geodesic space that satisfies the volume growth condition (\ref{growthcondition}). 
Therefore, since $CD^e(\kappa,N)$ implies $CD(K,\infty)$ the main result of 
\cite{agsriemannian} yields the existence of $EVI_K$-gradient flow curves. Additionally, in \cite{agmr} the authors prove that for ``good'' geodesics $\mu_t$ in $\mathcal{P}_2(X)$, one has
\begin{align*}
\frac{d^+}{ds}\frac{1}{2}W_2(\mu^s_0,\mu_1)^2\leq \frac{d}{dt}\Ent(\mu_t)\Big|_{t=0}.
\end{align*}
But $\mms$ already satisfies $CD(K,\infty)$ and has a quadratic Cheeger energy. Hence, it satisfies the condition $RCD(K,\infty)$ in the sense of \cite{agsriemannian}, 
Wasserstein geodesics in $\mathcal{P}_2(\m_{\sX})$ are unique, and 
therefore are good geodesic in the sense of \cite{agmr}. Then, we can copy the proof of Lemma \ref{alemma}.
\\
\\
``(iii) $\Rightarrow$ (ii)'': Since $\kappa$ is bounded from below and by monotonicity of 
\begin{align*}
\frac{d}{dt}\Big|_{t=0}\sigma_{\kappa_{}}^{\sscr{(t)}}(\theta)\ \&\ 
\frac{d}{dt}\Big|_{t=1}\sigma_{\kappa_{}}^{\sscr{(t)}}(\theta)
\end{align*} 
$\mms$ already satisfies $EVI_{K}$, and consequently it is infinitesimally Hilbertian by \cite{agsheat}.\medskip\\
We will prove the entropic curvature-dimension condition $CD^e(\kappa,N)$ following the proof of Theorem \ref{strong}.
But, recall that the entropic curvature dimension condition for variable $\kappa$ is not just $(\kappa,N)$-convexity of the entropy.
%, and $\mathcal{P}_2(X)$ is in general not locally compact.
%Therefore, we have to check carefully what is different, and use the right modifications. 
\smallskip\\
First, assume $\kappa$ is contiuous on $X$.
%that is assumed to be compact.
Pick a $L^2$-Wasserstein geodesic $\Pi$ in $\mathcal{P}(\mathcal{G}(X))$ and let $(e_t)_{\star}\Pi=\mu_t$.
Let $\bar{\mu}:[0,\Theta]\rightarrow X$ its 1-speed reparametrization. Note that
$\kappa_{\Pi}(t \Theta)\Theta^2=:K_{\Pi}(t)$
is just lower semi-continuous. Therefore we replace it by functions $K_{\Pi,n}:[0,1]\rightarrow \mathbb{R}$ that are continuous monotone non-decreasing and converge pointwise to $K_{\Pi}$.
Let $\delta>0$ be arbitrary.
Since $K_{\Pi,n}$ is continuous, we have that $K_{\Pi,n}(\cdot/\Theta)$ is uniformily continous on $[0,\Theta]$. Hence, we can find ${h}>0$ and points $r_i\in[0,\Theta]$ for $i=1,\dots,N$ such that
\begin{align*}
\max K_{\Pi}|_{B_{2{h}}(r_i)}-\min K_{\Pi}|_{B_{2{h}}(r_i)}<\delta
\end{align*}
for each $i=1,\dots,N$. Now, we pick $\hat{r}\in [0,\theta]$ and $\epsilon>0$, 
and consider $\bar{\gamma}=\bar{\mu}|_{[\hat{r}-\epsilon,\hat{r}+\epsilon]}$ such that $\hat{r}\pm\epsilon\in[r_i-{h},r_i+{h}]$ for some $i=1,\dots,N$.
Its constant speed reparametrization is $\gamma_{}:[0,1]\rightarrow \mathcal{P}_2(X)$.
Let $\nu^s$ be the $EVI_{\kappa,\sN}$ gradient flow curve starting in $\mu_{\textstyle{\frac{1}{2}}}$. Then, we obtain
\begin{align*}
-\frac{1}{2N}\frac{d}{ds}W_2(\mu^s,(e_0)_{\star}\Pi)^2+\frac{d}{dt}\sigma_{\kappa^-_{\tilde{\Pi}_0^s}/\sN}^{\sscr{(t)}}(\tilde{\Theta}_0^s)|_{t=1}\geq \frac{d}{dt}\sigma_{\kappa^+_{\tilde{\Pi}_0^s}/\sN}^{\sscr{(t)}}(\tilde{\Theta}_0^s)|_{t=0}\frac{U_N(\gamma_{}(0))}{U_N(\mu^s)}
\end{align*}
where $\tilde{\Pi}_0^s$ and $\tilde{\Pi}_1^s$ are geodesics between 
$(e_0)_{\star}\Pi$ and $\nu^s$, and $(e_1)_{\star}\Pi$ and $\nu^s$ respectively, and $\tilde{\Theta}_0^s=W_2((e_0)_{\star}\Pi,\nu^s)$ and $\tilde{\Theta}_1^s=W_2((e_1)_{\star}\Pi,\nu^s)$. Local compactness of $X$ yields weak convergence of $\tilde{\Pi}^s_{0/1}$ for $s\rightarrow 0$.
\smallskip\\
By lower semi-continuity of the $L^2$-Wasserstein distance the limits $\tilde{\Pi}_0$ and $\tilde{\Pi}_1$ are geodesic between $(e_0)_{\star}\Pi$ and $(e_{\hat{r}})_{\star}\Pi$, and $(e_1)_{\star}\Pi$ and $(e_{\hat{r}})_{\star}\Pi$ respectively. 
Additionally, 
the lower semi-continuity yields that the concatenation of the geodesics $\tilde{\Pi}_0$ and $\tilde{\Pi}_1^-$ as absolutely continuous curves in $\mathcal{P}_2(X)$ w.r.t. $W_2$ is a geodesic as well. But
we know that $X$ already satisfies a condition $RCD^*(K,\infty)$ for some $K$. Hence, $L^2$-Wasserstein geodesics between absolutely continuous probability measures are unique, and therefore we have $\Pi=\tilde{\Pi}$.
\smallskip\\
As in the proof of Theorem \ref{strong} we obtain a weak differential inequality for $U_N$ along $\Pi$, $K_{\Pi,n}$ and $\delta$. By standard convergence results and monotonicity properties the statement follows 
as in the proof of Theorem \ref{strong}.
\end{proof}
\begin{theorem}\label{C}
Let $\mmsi_{i\in\mathbb{N}}$ be a sequence of metric measure spaces with $\m_{\sX_i}<\infty$ converging in Gromov sense to a metric measure space $\mms$. 
Let $\kappa_i:X_i\rightarrow \mathbb{R}$ be lower semi-continuous functions such that $\mmsi$ satisfies the condition $RCD^*(\kappa_i,N_i)$.
Additionally, consider an admissible function $\VK:X\rightarrow \mathbb{R}$ and $N\in[1,\infty)$ such that 
\begin{align}\label{hateit}
\liminf_{i\rightarrow\infty}\VK_i\geq \VK\geq K\in\mathbb{R}\ \ \& \ \ \limsup_{i\rightarrow \infty}N_i\leq N
\end{align}
Then $\mms$ satisfies the condition $RCD^*(\kappa,N)$.
\end{theorem}
\begin{proof}
Since $\kappa_i$ and $\kappa$ are bounded from below by a constant $K$, $\mms$ already satisfies the condition $RCD^*(K,N)$.
Then, by combination of Theorem \ref{mmmm} and Theorem \ref{B} the result follows. 
\end{proof}
\begin{corollary}\label{D}
Let $(M_i,g_{\sM_i})_{i\in\mathbb{N}}$ be a family of compact Riemannian manifolds such that $\ric_{\sM_i}\geq \VK_i\ \& \ \dim_{\sM_i}\leq N$ where $\VK_i:M_i\rightarrow \mathbb{R}$ is a family of equi-continuous functions such that
$\VK_i\geq -C$ for some $C>0$. 
There exists subsequence of $(M_i,\de_{\sM_i},\vol_{\sM_i})$ that converges in 
measured Gromov-Hausdorff sense to a metric measure space $(X,\de_{\sX},\m_{\sX})$, and there exists a subsequence of $\VK_i$ such that $\lim \VK_i= \VK$.
Then $X$ satisfies the condition $RCD^*(\VK,N)$.
\end{corollary}
\begin{proof}
Since there is uniform lower bound for the Ricci curvature, Gromov's compactness theorem yields a converging subsequence. Then, Gromov's Arzela-Ascoli theorem also yields a uniformily converging subsequence of $\VK_i$ with limit $\VK$.
Finally, if we apply the previous stability theorem, we obtain the result. 
\end{proof}

\section{Wasserstein contraction}
\noindent 
From $EVI_{\kappa,\infty}$ one can deduce easily a Wasserstein contraction estimate.
\begin{theorem}\label{control} Let $\mms$ be a metric measure spaces satisfying $EVI_{\kappa,\N}$ where $\kappa:X\rightarrow \mathbb{R}$ is lower semi-continuous.
Consider Wasserstein $EVI_{\kappa,\sN}$-gradient flow curves $\mu^s$ and $\nu^s$ with initial measures $\mu$ and $\nu$. Let $\Pi^s$ be the $L^2$-Wasserstein geodesic between $\mu^s$ and $\nu^s$.
Then the following contraction estimate holds
\begin{align*}
\frac{d^+}{d s}W_2(\mu_s,\nu_s)^2
\leq -2\int_0^1\int\kappa(\gamma(t))|\dot{\gamma}|^2d\Pi(\gamma)^sdt.
\end{align*}
\end{theorem}
\begin{proof} Note that by lower semi-conitnuity of $\kappa$ $\mms$ satisfies a condition $RCD^*(K,N)$ for some constant $K$. In particular, Wasserstein geodesics between measures in $\mathcal{P}^2(\m_{\sX})$ are unique.
Consider $s_0,s_1\in [0,\infty)$.
In (\ref{evi}) we set $\mu^s=\mu^s$ and $\nu=\nu^{s_1}$. Integration from $s_0$ to $s_1$ in $s$ yields
\begin{align*}
\frac{1}{2}W_2(\mu_{s_1},\nu_{s_1})^2-\frac{1}{2}W_2(\mu_{s_0},\nu_{s_1})^2&+\int_{s_0}^{s_1}\left[\int_0^1 (1-t)\kappa_{\sPi^s}(t\Theta^s)dt\right]W_2(\mu^s,\mu^{s_1})^2ds \\ 
&\leq\left[\Ent(\mu_{s_0})-\Ent(\nu_{s_1})\right](s_1-s_0)
\end{align*}
where $\Pi^s$ is the optimal dynamical plan between $\mu^s$ and $\nu$. 
We used that $\Ent$ is monotone decreasing along gradient flow curves. Similar, if we put $\nu=\mu^{s_0}$ and $\nu^s=\mu^{s}$ and integrate again from $s_0$ to $s_1$, then we obtain
\begin{align*}
\frac{1}{2}W_2(\mu_{s_0},\nu_{s_1})^2-\frac{1}{2}W_2(\mu_{s_0},\nu_{s_0})^2&+\int_{s_0}^{s_1}\left[\int_0^1 (1-t)\kappa_{\underline{\sPi}^s}(t\Theta^s)dt\right]W_2(\mu^{s_0},\nu^{s})^2ds\\
&\leq (\Ent(\nu_{s_0})-\Ent(\mu_{s_0}))(s_1-s_0)
\end{align*}
where $\underline{\Pi}^s$ is the optimal dynamical plan between $\nu^s$ and $\mu^{s_0}$.
Adding the last two inequalities, deviding by $s_{1}-s_{0}$ and letting $s_1\rightarrow s_0$ yields
\begin{align*}
&\frac{d^+}{d s}|_{s_0}\frac{1}{2}W_2(\mu_s,\nu_s)^2\\
&\leq \left[-\int_0^1 (1-t)\kappa_{\sPi^{s_0}}(t\Theta^{s_0})dt-\int_0^1 (1-t)\kappa_{\underline{\sPi}^{s_0}}(t\Theta^{s_0})dt\right] W_2(\mu_{s_0},\nu_{s_0})^2.
\end{align*}
Since there is a unique optimal dynamical plan between $\nu^{s_0}$ and $\mu^{s_0}$, we have that $\Pi^{s_0}=\underline{\Pi}^{s_0,-}$.
Therefore
\begin{align*}
\int_0^1 (1-t)\kappa_{\underline{\sPi}^{s_0}}(t\Theta^{s_0})dt=\int_0^1 (1-t)\kappa_{{\sPi}^{s_0}}((1-t)\Theta^{s_0})dt=\int_0^1 t\kappa_{{\sPi}^{s_0}}(t\Theta^{s_0})dt
\end{align*}
Hence
\begin{align*}
\frac{d^+}{d s}\frac{1}{2}W_2(\mu_s,\nu_s)^2
\leq \left[-\int_0^1 \kappa_{{\sPi}^{s}}(t\Theta^{s})dt\right] W_2(\mu_{s},\nu_{s})^2=-\int_0^1\int\kappa(\gamma(t))|\dot{\gamma}|^2d\Pi(\gamma)^sdt.
\end{align*}
\end{proof}
\begin{remark}Following the same lines as in the proof of previous theorem one obtains the following.
If $(X,\de_{\sX})$ is a locally compact complete length space with unique geodesics and a $\kappa$-convex function $f:X\rightarrow[0,\infty)$, we can deduce
\begin{align*}
\frac{d^+}{d s}\de_{\sX}(x_s,y_s)^2
\leq -2\int_0^1\kappa(\gamma^s(t))dt \de_{\sX}(x_0,y_0)^2
\end{align*}
where $x_s$ and $y_s$ are $EVI_{\kappa}$-gradient flow curves of $f$. Then, an application of Gromwall's lemma yields
\begin{align*}
\de_{\sX}(x_s,y_s)^2\leq e^{-2\int_0^s \int_0^1\kappa(\gamma^{\tau}(t))dtd\tau}\de_{\sX}(x_0,y_0)^2.
\end{align*}

\end{remark}
\paragraph{\textbf{The case $N<\infty$}} First, we deduce a contraction estimate for $EVI_{\kappa,\sN}$-gradient flow curves $x_s$ and $y_s$ for $f$ on a metric space $(X,\de_{\sX})$ as in the previous remark.
Consider
\begin{align*}
-\frac{1}{2N}\frac{d}{ds}\de_{\sX}(x_s,z)^2+\frac{d}{dt}\sigma_{\kappa^-_{\gamma^s}/\sN}^{\sscr{(t)}}(|\dot{\gamma}^s|)|_{t=1}\geq \frac{d}{dt}\sigma_{\kappa^+_{\gamma^s}/\sN}^{\sscr{(t)}}(|\dot{\gamma}^s|)|_{t=0}\frac{U_N(z)}{U_N(x_s)}
\end{align*}
and rewrite as follows
\begin{align}\label{mm}
&\frac{1}{2N}\frac{d}{ds}\de_{\sX}(x_s,z)^2+\frac{d}{dt}\Big|_{t=0}\underbrace{\left[\sigma_{\kappa^-_{\gamma^s}/\sN}^{\sscr{(1-t)}}(|\dot{\gamma}^s|)+\sigma_{\kappa^+_{\gamma^s}/\sN}^{\sscr{(t)}}(|\dot{\gamma}^s|)\right]}_{=:w}\nonumber\\
&\hspace{6cm}\leq \frac{d}{dt}\sigma_{\kappa^+_{\gamma^s}/\sN}^{\sscr{(t)}}(|\dot{\gamma}^s|)|_{t=0}\left[1-\frac{U_N(z)}{U_N(x_s)}\right]
\end{align}
$w$ solves $w''+ |\dot{\gamma}^s|\frac{\kappa\circ\gamma}{N} w=0$ with $w(0)=w(1)=1$. Therefore
\begin{align*}
&\frac{d}{dt}\left[\sigma_{\kappa^-_{\gamma^s}/\sN}^{\sscr{(1-t)}}(|\dot{\gamma}^s|)+\sigma_{\kappa^+_{\gamma^s}/\sN}^{\sscr{(t)}}(|\dot{\gamma}^s|)\right]_{t=0}
%&=\int_0^1(1-\tau)\kappa(\gamma(\tau))|\dot{\gamma}|^2 \left[\sigma_{\kappa^-_{\gamma^s}}^{\sscr{(1-\tau)}}(|\dot{\gamma}^s|)+\sigma_{\kappa^+_{\gamma^s}}^{\sscr{(\tau)}}(|\dot{\gamma}^s|)\right]d\tau
\\
&\hspace{3cm}=\underbrace{\int_0^1(1-\tau)\frac{\kappa(\gamma(\tau))}{N}\left[\sigma_{\kappa^-_{\gamma^s}/\sN}^{\sscr{(1-\tau)}}(|\dot{\gamma}^s|)+\sigma_{\kappa^+_{\gamma^s}/\sN}^{\sscr{(\tau)}}(|\dot{\gamma}^s|)\right]d\tau}_{=a_{\gamma^s}}|\dot{\gamma}|^2 
\end{align*}
Hence, we can rewrite the left hand side of (\ref{mm}) as follows
\begin{align*}\frac{1}{2N}
e^{-2N\int_0^s a_{\gamma^s}ds}\frac{d}{ds}\left[e^{\int_0^s 2Na_{\gamma^s}ds}\de_{\sX}(x_s,z)^2\right].
\end{align*}Then
(\ref{mm}) becomes
\begin{align*}\frac{1}{2N}
\frac{d}{ds}e^{2N\int_0^sa_{\gamma^s}ds}\de_{\sX}(x_s,z)^2\leq e^{2N\int_0^s a_{\gamma^s}ds}\frac{|\dot{\gamma}^s|}{\frs_{\kappa_{\gamma^s}^+/\sN}(|\dot{\gamma}|)}\left[1-\frac{U_N(z)}{U_N(x_s)}\right]
\end{align*}
and integration with respect to $s$ from $s_1$ to $s_2$ yields
\begin{align*}
\frac{1}{2N}e^{2N\int_{s_1}^{s_2}a_{\gamma^s}ds}\de_{\sX}(x_{s_2},z)^2-\frac{1}{2N}\de_{\sX}(x_{s_1},z)^2\leq \int_{s_1}^{s_2}e^{2N\int_{s_1}^s a_{\gamma^s}ds}\frac{|\dot{\gamma}^s|}{\frs_{\kappa_{\gamma^s}^+/\sN}(|\dot{\gamma}|)}\left[1-\frac{U_N(z)}{U_N(x_s)}\right]ds.
\end{align*}
Since $s\mapsto \frac{|\dot{\gamma}^s|}{\frs_{\kappa_{\gamma^s}^+/N}(|\dot{\gamma}|)}$ is continuous, and since $U_{\sN}$ is increasing w.r.t. $x_s$, 
the right hand side can be estimated by
\begin{align*}
%&\int_{s_1}^{s_2}e^{\int_0^s a_{\gamma^s}ds}\frac{|\dot{\gamma}^s|}{\frs_{\kappa_{\gamma^s}^+}(|\dot{\gamma}|)}Nds-\int_{s_1}^{s_2}e^{\int_0^s a_{\gamma^s}ds}\frac{|\dot{\gamma}^s|}{\frs_{\kappa_{\gamma^s}^+}(|\dot{\gamma}|)}N\frac{U_N(z)}{U_N(x_s)}ds\\
%&
\leq  \underbrace{\max_{s\in[s_1,s_2]}\frac{|\dot{\gamma}^s|}{\frs_{\kappa_{\gamma^s}^+/\sN}(|\dot{\gamma}|)}}_{M_{\gamma}(s_1,s_2)}\int_{s_1}^{s_2}e^{2N\int_{s_1}^s a_{\gamma^s}ds}ds-\underbrace{\min_{s\in[s_1,s_2]}\frac{|\dot{\gamma}^s|}{\frs_{\kappa_{\gamma^s}^+/\sN}(|\dot{\gamma}|)}}_{m_{\gamma}(s_1,s_2)}\int_{s_1}^{s_2}e^{2N\int_{s_1}^s a_{\gamma^s}ds}\frac{U_N(z)}{U_N(x_{s_2})}ds.
\end{align*}
It follows
\begin{align*}
&m_{\gamma}(s_1,s_2)\frac{U_N(z)}{U_N(x_{s_2})}\leq  {M_{\gamma}(s_1,s_2)}\\
&-{\left[2N\int_{s_1}^{s_2}e^{2N\int_{s_2}^s a_{\gamma^s}ds}\right]^{-1}}\de_{\sX}(x_{s_2},z)^2+{\left[2N\int_{s_1}^{s_2}e^{2N\int_{s_1}^s a_{\gamma^s}ds}\right]^{-1}}\de_{\sX}(x_{s_1},z)^2
\end{align*}
\smallskip\\
%Now, let us fix $s$ and $s'$, consider $evi$-
Consider gradient flow curves $x_s$ and $y_s$ and choose $\lambda, r>0$.
%such that $\lambda r=s'$ and $\lambda^{-1} r=s$, i.e. $\sqrt{s'/s}=\lambda$ and $\sqrt{s's}=r$. 
We apply the previous inequality for
$z=y_{\lambda^{-1}r}$ and $s_1=\lambda r$ and $s_2=\lambda(r+\epsilon)$ for some $\epsilon>0$.
\begin{align*}
&m_{\gamma}(\lambda r,\lambda(r+\epsilon))
\frac{U_N(y_{\lambda^{-1}r})}{U_N(x_{\lambda(r+\epsilon)})}\leq  M_{\gamma}(\lambda r,\lambda(r+\epsilon))\\
&-\left[2N\int_{\lambda r}^{\lambda(r+\epsilon)}e^{-2N\int^{\lambda(r+\epsilon)}_s a_{\gamma^s}ds}\right]^{-1}\de_{\sX}(x_{\lambda(r+\epsilon)},y_{\lambda^{-1}r})^2\\
&+\left[2N\int_{\lambda r}^{\lambda(r+\epsilon)}e^{2N\int_{\lambda r}^s a_{\gamma^s}ds}\right]^{-1}\de_{\sX}(x_{\lambda r},y_{\lambda^{-1}r})^2
\end{align*}
And similar if we switch the roles of $x_s$ and $y_s$, and if we set $z=x_{\lambda(r+\epsilon)}$ and $s_1=\lambda^{-1}r, s_2=\lambda^{-1}(r+\epsilon)$.
\begin{align*}
&m_{\tilde{\gamma}}(\lambda^{-1}r, \lambda^{-1}(r+\epsilon))\frac{U_N(x_{\lambda(r+\epsilon)})}{U_N(y_{\lambda^{-1}(r+\epsilon)})}\leq  M_{\tilde{\gamma}}(\lambda^{-1}r,\lambda^{-1}(r+\epsilon))\\
&-\left[2N\int_{\lambda^{-1}r}^{\lambda^{-1}(r+\epsilon)}e^{-2N\int^{\lambda^{-1}(r+\epsilon)}_s a_{\tilde{\gamma}^s}ds}\right]^{-1}\de_{\sX}(y_{\lambda^{-1}(r+\epsilon)},x_{\lambda(r+\epsilon)})^2\\
&+\left[2N\int_{\lambda^{-1}r}^{\lambda^{-1}(r+\epsilon)}e^{2N\int_{\lambda^{-1}r}^s a_{\tilde{\gamma}^s}ds}\right]^{-1}\de_{\sX}(y_{\lambda^{-1}r},x_{\lambda(r+\epsilon)})^2
\end{align*}
where $\tilde{\gamma}^s$ is the geodesic between $y^s$ and $z$.
We set
\begin{align*}
\int_{\lambda r}^{\lambda(r+\epsilon)}e^{-2N\int^{\lambda(r+\epsilon)}_s a_{\gamma^s}ds}=:e(\lambda, \epsilon, -a_{\gamma^s}), \int_{\lambda r}^{\lambda(r+\epsilon)}e^{2N\int_{\lambda r}^s a_{\gamma^s}ds}=:e(\lambda,\epsilon,a_{\gamma^s}).
\end{align*}
If we multiply the resulting two formulas, take sqareroots and use Young's inequality $2\sqrt{ab}\leq \lambda a+\lambda^{-1}b$ for $\lambda$ as before, we obtain
\begin{align*}
&2N 2\sqrt{m_{\gamma}(\lambda r,\lambda(r+\epsilon))\frac{U_N(y_{\lambda^{-1}r})}{U_N(y_{\lambda{-1}(r+\epsilon)})}m_{\tilde{\gamma}}(\lambda^{-1} r,\lambda^{-1}(r+\epsilon))}\\
&\leq 2N\left[ M_{\gamma}(\lambda r,\lambda(r+\epsilon))\lambda^{-1}+M_{\tilde{\gamma}}(\lambda^{-1} r,\lambda^{-1}(r+\epsilon))\lambda\right]\\
&\hspace{0.5cm} + \de_{\sX}(y_{\lambda^{-1}r},x_{\lambda(r+\epsilon)})^2\left[\frac{\lambda^{-1}}{e(\lambda^{-1},\epsilon,a_{\tilde{\gamma}^s})}-\frac{\lambda}{e(\lambda,\epsilon,-a_{{\gamma^s}})}\right]\\
& \hspace{0.5cm}+ \de_{\sX}(x_{\lambda r},y_{\lambda^{-1}r})^2\left[\frac{\lambda}{e(\lambda,\epsilon,a_{{\gamma}^s})}-\frac{\lambda^{-1}}{e(\lambda^{-1},\epsilon,-a_{\tilde{\gamma}^{s}})}\right]\\
&\hspace{0.5cm} - \frac{\lambda^{-1}\epsilon}{e(\lambda^{-1},\epsilon, -a_{\tilde{\gamma}^s})}\frac{1}{\epsilon}\left[\de_{\sX}(y_{\lambda^{-1}(r+\epsilon)},x_{\lambda(r+\epsilon)})^2- \de_{\sX}(y_{\lambda^{-1}r},x_{\lambda r})^2\right].
\end{align*}
Now, let $\epsilon\rightarrow 0$. First. note that
\begin{align*}
\frac{\lambda^{-1}\epsilon}{e(\lambda^{-1},\epsilon, -a_{\tilde{\gamma}^s})}\rightarrow 1
\end{align*}
and 
\begin{align*}
\frac{\lambda^{-1}}{e(\lambda^{-1},\epsilon,a_{\tilde{\gamma}^s})}-\frac{\lambda}{e(\lambda,\epsilon,-a_{{\gamma^s}})}\rightarrow -N(\lambda^{-1}a_{\tilde{\gamma}^{\lambda^{-1}r}}+\lambda a_{{\gamma}^{\lambda r}} )\\
\frac{\lambda}{e(\lambda,\epsilon,a_{{\gamma}^s})}-\frac{\lambda^{-1}}{e(\lambda^{-1},\epsilon,-a_{\tilde{\gamma}^{s}})}\rightarrow -N(\lambda^{-1}a_{\tilde{\gamma}^{\lambda^{-1}r}}+\lambda a_{{\gamma}^{\lambda r}} )
\end{align*}
Also note, that 
\begin{align*}
m_{\gamma}(\lambda r,\lambda(r+\epsilon)),\ \ M_{\gamma}(\lambda r,\lambda(r+\epsilon))\rightarrow \frac{|\dot{\gamma}^{\lambda r}|}{\frs_{\kappa_{\gamma^{\lambda r}}/\sN}(|\dot{\gamma}^{\lambda r}|)}
\end{align*}
$\gamma^{\lambda r}$ is the unique geodesic between $x_{\lambda r}$ and $y_{\lambda^{-1}r}$. Therefore $\gamma^{\lambda r}=(\tilde{\gamma}^{\lambda^{-1}r})^-$.
And since 
\begin{align*}
a_{\tilde{\gamma}^{\lambda^{-1}r}}&=\int_0^1(1-\tau)\frac{\kappa(\tilde{\gamma}(\tau))}{N}\left[\sigma_{\kappa^-_{\tilde{\gamma}^{\lambda^{-1}r}}/\sN}^{\sscr{(1-\tau)}}(|\dot{\tilde{\gamma}}^{\lambda^{-1}r}|)+\sigma_{\kappa^+_{\tilde{\gamma}^{\lambda^{-1}r}}/\sN}^{\sscr{(\tau)}}(|\dot{\tilde{\gamma}}^{\lambda^{-1}r}|)\right]d\tau\\
&=\int_0^1\tau\frac{\kappa({\gamma}(\tau))}{N}\left[\sigma_{\kappa^-_{{\gamma}^{\lambda^{-1}r}}/\sN}^{\sscr{(1-\tau)}}(|\dot{{\gamma}}^{\lambda^{-1}r}|)+\sigma_{\kappa^+_{{\gamma}^{\lambda^{-1}r}}/\sN}^{\sscr{(\tau)}}(|\dot{{\gamma}}^{\lambda^{-1}r}|)\right]d\tau
\end{align*}
we obtain, that
\begin{align*}
&N\left[\lambda a_{\gamma^{\lambda^{}r}}+\lambda^{-1}a_{\tilde{\gamma}^{\lambda^{-1} r}}\right]\\
&=\int_0^1{\kappa(\gamma(\tau))}\left[\left((1-\tau)\lambda+\tau\lambda^{-1}\right)\left[\sigma_{\kappa^-_{\gamma^{\lambda r}}/\sN}^{\sscr{(1-\tau)}}(|\dot{\gamma}^{\lambda r}|)+\sigma_{\kappa^+_{\gamma^{\lambda r}}/\sN}^{\sscr{(\tau)}}(|\dot{\gamma}^{\lambda r}|)\right]\right]d\tau=:b_{\gamma^{\lambda r}}
\end{align*}
Therefore, if we set $g(r)=\de_{\sX}(y_{\lambda^{-1}r},x_{\lambda r})^2$, we obtain
\begin{align*}
&\frac{d}{dr}g(r)\leq -2b_{\gamma^{\lambda r}}g(r) \\
&+ 2N
\left[\frac{\lambda}{\frs_{\kappa_{\gamma^{\lambda r}}/\sN}(|\dot{\gamma}^{\lambda r}|)} +\frac{\lambda^{-1}}{\frs_{\kappa_{\gamma^{\lambda r}}^-/\sN}(|\dot{\gamma}^{\lambda r}|)} -\frac{2}{\sqrt{\frs_{\kappa_{\gamma^{\lambda r}}/\sN}(|\dot{\gamma}^{\lambda r}|)\frs_{\kappa_{\gamma^{\lambda r}}^-/\sN}(|\dot{\gamma}^{\lambda r}|)}}\right]|\dot{\gamma}^{\lambda r}|
\end{align*}
%This can be reformulated as
%\begin{align*}
%\frac{d}{dr}g(r)\leq  b_{\gamma^{\lambda r}}g(r) + 
%N\left[\textstyle{\sqrt{\frac{\lambda |\dot{\gamma}^{\lambda r}|}{\frs_{\kappa_{\gamma^{\lambda r}}/\sN}(|\dot{\gamma}^{\lambda r}|)}} -\sqrt{\frac{\lambda^{-1} |\dot{\gamma}^{\lambda r}|}{\frs_{\kappa_{\gamma^{\lambda r}}^-/\sN}(|\dot{\gamma}^{\lambda r}|)}}}\right]^2
%\end{align*}
%
%
%
%\begin{align*}
%\frac{d}{dr}g(r)\leq  b_{\gamma^{\lambda r}}g(r) + \frac{1}{N}
%\left[\textstyle{\sqrt{\frac{\lambda |\dot{\gamma}^{\lambda r}|}{\frs_{\kappa_{\gamma^{\lambda r}}}(|\dot{\gamma}^{\lambda r}|)}} -\sqrt{\frac{\lambda^{-1} |\dot{\gamma}^{\lambda r}|}{\frs_{\kappa_{\gamma^{\lambda r}}^-}(|\dot{\gamma}^{\lambda r}|)}}}\right]^2
%\end{align*}
or equivalently
\begin{align}\label{remember}
\frac{d}{dr}g(r)\leq - 2b_{\gamma^{\lambda r}}g(r) + 2N
\left[\textstyle{\sqrt{\lambda \frac{d}{dt}\big|_{t=0}\sigma_{\kappa_{\gamma}/N}^{(t)}(|\dot{\gamma}^{\lambda r}|)} -\sqrt{\lambda^{-1} \frac{d}{dt}\big|_{t=0}\sigma_{\kappa_{\gamma}^-/N}^{(t)}(|\dot{\gamma}^{\lambda r}|)}}\right]^2
\end{align}
\begin{remark}
If we set $\lambda =1$, this becomes
\begin{align}\label{imagine}
\frac{d}{dr}\de_{\sX}(y_{s},x_{s})^2&\leq -2\int_0^1{\kappa(\gamma(\tau))}{\left[\sigma_{\kappa^-_{\gamma^{s}}/\sN}^{\sscr{(1-\tau)}}(|\dot{\gamma}^{s}|)+\sigma_{\kappa^+_{\gamma^{s}}/\sN}^{\sscr{(\tau)}}(|\dot{\gamma}^{s}|)\right]}d\tau
\de_{\sX}(y_{s},x_{s})^2\nonumber\\
&\hspace{1cm}+ 2N
\left[\textstyle{\sqrt{ \frac{d}{dt}\big|_{t=0}\sigma_{\kappa_{\gamma}}^{(t)}(|\dot{\gamma}^{s}|)} -\sqrt{\frac{d}{dt}\big|_{t=0}\sigma_{\kappa_{\gamma}^-}^{(t)}(|\dot{\gamma}^{s}|)}}\right]^2
\end{align}
If $\kappa=K$ is constant, then (\ref{imagine}) simplifies to
\begin{align*}
\frac{d}{dr}\de_{\sX}(y_{s},x_{s})^2&\leq -2K\int_0^1{\left[\sigma_{\sK/\sN}^{\sscr{(1-\tau)}}(\de_{\sX}(x_s,y_s))+\sigma_{\sK/\sN}^{\sscr{(\tau)}}(\de_{\sX}(x_s,y_s))\right]}d\tau
\de_{\sX}(y_{s},x_{s})^2\end{align*}
And if $N\rightarrow \infty$, (\ref{imagine}) becomes
\begin{align*}
&\frac{d}{dr}\de_{\sX}(y_{s},x_{s})^2\leq -2\int_0^1{\kappa(\gamma(\tau))} d\tau
\de_{\sX}(y_{s},x_{s})^2
\end{align*}
%This follows, since we can rewrite the last term in (\ref{remember})
%\begin{align*}
%&N{\scriptstyle \left[\textstyle{\sqrt{\lambda \frac{d}{dt}\big|_{t=0}\sigma_{\kappa_{\gamma}/N}^{(t)}(|\dot{\gamma}^{\lambda r}|)} -\sqrt{\lambda^{-1} \frac{d}{dt}\big|_{t=0}\sigma_{\kappa_{\gamma}^-/N}^{(t)}(|\dot{\gamma}^{\lambda r}|)}}\right]^2}\\
%&={\scriptstyle\left[\textstyle{\sqrt{\lambda N\frac{d}{dt}\big|_{t=0}\left(\sigma_{\kappa_{\gamma}/N}^{(t)}(|\dot{\gamma}^{\lambda r}|)-t\right)+\lambda N} -\sqrt{N\lambda^{-1} \frac{d}{dt}\big|_{t=0}\left(\sigma_{\kappa_{\gamma}^-/N}^{(t)}(|\dot{\gamma}^{\lambda r}|)-t\right)+\lambda^{-1} N}}\right]^2}.
%\end{align*}
%Note that 
%\begin{align*}
%N\frac{d}{dt}\big|_{t=0}\left(\sigma_{\kappa_{\gamma}/N}^{(t)}(|\dot{\gamma}^{\lambda r}|)-t\right)=\int_0^1(1-\tau)\kappa(\gamma(\tau))\sigma_{\kappa_{\gamma}/N}^{(\tau)}(|\dot{\gamma}^{\lambda r}|)d\tau \rightarrow \int_0^1(1-\tau)\tau\kappa(\gamma(\tau)) d\tau.
%\end{align*}
%Therefore, the statement follows. 
If $\lambda\neq 1$, the second term in the right hand side in (\ref{remember}) tends to $\infty$ for $N\rightarrow \infty$.
\end{remark}
If we follow the same reasoning, in the context of Wasserstein $EVI_{\kappa,\sN}$-gradient flow curves for a compact metric measure spaces, we obtain the next theorem.
\begin{theorem}\label{dimensionalcontraction}
Let $\mms$ be a compact metric measure spaces satisfying the condition $RCD^*(\kappa,N)$ where $\kappa:X\rightarrow \mathbb{R}$ is lower semi-continuous.
Consider Wasserstein $EVI_{\kappa,\sN}$-gradient flow curves $\mu^s$ and $\nu^s$ with initial measures $\mu$ and $\nu$.
Let $\lambda, r>0$.
%such that $\lambda r=s'$ and $\lambda^{-1} r=s$, i.e. $\sqrt{s'/s}=\lambda$ and $\sqrt{s's}=r$. 
Then the following contraction estimate holds
\begin{align}\label{woman}
&\frac{d}{dr} W_2(\mu^{\lambda r},\nu^{\lambda^{-1}r})^2\nonumber\\
&\hspace{10pt}\leq-
2\int_0^1{\kappa(\gamma(\tau))}\left[\left((1-\tau)\lambda+\tau\lambda^{-1}\right)\Big(\sigma_{\kappa^-_{\Pi^{\lambda r}}\Theta_{\lambda r}^2/\sN}^{\sscr{(1-\tau)}}+\sigma_{\kappa^+_{\Pi^{\lambda r}}\Theta_{\lambda r}^2/\sN}^{\sscr{(\tau)}}\Big)\right]d\tau\nonumber\\
&\hspace{20pt}+ 2N
\left[\textstyle{\sqrt{\lambda \frac{d}{dt}\big|_{t=0}\sigma_{\kappa^+_{\Pi^{\lambda r}}\Theta_{\lambda r}^2/\sN}^{\sscr{(\tau)}}} -\sqrt{\lambda^{-1} \frac{d}{dt}\big|_{t=0}\sigma_{\kappa^-_{\Pi^{\lambda r}}\Theta_{\lambda r}^2/\sN}^{\sscr{(\tau)}}}}\right]^2
\end{align}
where $\Pi^{\lambda r}$ is the unique $L^2$-Wasserstein geodesic between $\mu^{\lambda r}$ and $\nu^{\lambda^{-1} r}$.
\end{theorem}
\begin{remark}
If we consider $\kappa=K$ constant, in the limit $W_2(\mu_0,\nu_0)\rightarrow 0$ the right hand side of (\ref{woman}) is
\begin{align}
\sim-K\left(\lambda+\lambda^{-1}\right)W_2(\mu_0,\nu_0)^2 + 2N\left[\sqrt{\lambda} -\sqrt{\lambda^{-1}}\right]^2.
\end{align}
That is the same assymptotic behaviour as the corresponding contraction estimates in \cite{erbarkuwadasturm} (Remark 2.20) or in \cite{kuwadaspacetime}.
Hence, in Wasserstein space context with constant lower curvature bound our estimate yields the corresponding Bakry-Ledoux gradient estimate \cite{erbarkuwadasturm, kuwadaspacetime, bakryledouxliyau}.
\end{remark}
\small{
\bibliography{new}

\bibliographystyle{amsalpha}
% \nocite{*}
%\bibliography{Dissertation}
%\addcontentsline{toc}{chapter}{References}
}
\end{document}